\documentclass[11pt]{article}

\usepackage{amssymb,amsmath}
\usepackage{amsthm}
\usepackage{hyperref}
\usepackage{mathrsfs}
\usepackage{textcomp}

\usepackage{verbatim}

\usepackage{sectsty}

\makeatletter

\newdimen\bibspace
\renewenvironment{thebibliography}[1]{%
 \section*{\refname %or \bibname if you use ``book'' as the documentclass
       \@mkboth{\MakeUppercase\refname}{\MakeUppercase\refname}}%
     \list{\@biblabel{\@arabic\c@enumiv}}%
          {\settowidth\labelwidth{\@biblabel{#1}}%
           \leftmargin\labelwidth
           \advance\leftmargin\labelsep
           \itemsep\bibspace
           \parsep\z@skip     %
           \@openbib@code
           \usecounter{enumiv}%
           \let\p@enumiv\@empty
           \renewcommand\theenumiv{\@arabic\c@enumiv}}%
     \sloppy\clubpenalty4000\widowpenalty4000%
     \sfcode`\.\@m}
    {\def\@noitemerr
      {\@latex@warning{Empty `thebibliography' environment}}%
     \endlist}

\makeatother

\makeatletter

\newtheorem{thm}{Theorem}[section]
\newtheorem{lem}{Lemma}[section]
\newtheorem{prop}{Proposition}[section]
\newtheorem{defn}{Definition}[section]

\newtheorem{cor}{Corollary}[section]

\numberwithin{equation}{section}

\def\Xint#1{\mathchoice
  {\XXint\displaystyle\textstyle{#1}}%
  {\XXint\textstyle\scriptstyle{#1}}%
  {\XXint\scriptstyle\scriptscriptstyle{#1}}%
  {\XXint\scriptscriptstyle\scriptscriptstyle{#1}}%
  \!\int}
\def\XXint#1#2#3{{\setbox0=\hbox{$#1{#2#3}{\int}$}
  \vcenter{\hbox{$#2#3$}}\kern-.5\wd0}}

\def\dashint{\Xint-}

\newcommand{\al}{\alpha}                \newcommand{\lda}{\lambda}
                \newcommand{\pa}{\partial}
\newcommand{\va}{\varepsilon}           \newcommand{\ud}{\mathrm{d}}
\newcommand{\be}{\begin{equation}}      \newcommand{\ee}{\end{equation}}
                 
\newcommand{\Lda}{\Lambda}              
\newcommand{\R}{\mathbb{R}}              \newcommand{\Sn}{\mathbb{S}^n}
  \newcommand{\M}{\mathscr{M}}

\newcommand{\dsum}{\displaystyle\sum}

\begin{document}

\title{\textbf{On a fractional Nirenberg problem, part II: existence  of solutions}
\bigskip}

\author{\medskip Tianling Jin, \ \
YanYan Li, \ \
Jingang Xiong}

\date{\today}

\maketitle

\begin{abstract} This paper is a continuation of our earlier work
``[T. Jin, Y.Y. Li and J. Xiong, On a fractional Nirenberg problem, part I: blow up analysis and compactness of solutions, to appear in
J. Eur. Math. Soc.]", where compactness  results were given on a fractional Nirenberg problem.
We prove two existence results stated there.  We also obtain a fractional Aubin inequality.
\end{abstract}

\section{Introduction}
\medskip

\noindent

Let $(\Sn,g_{\Sn})$, $n\geq 2$, be the standard sphere in $\R^{n+1}$. The fractional Nirenberg problem studied in \cite{JLX} is equivalent to solving:
\be\label{main equ}
P_\sigma(v)=c(n,\sigma)K v^{\frac{n+2\sigma}{n-2\sigma}}\quad \mbox{on } \Sn,
\ee
where $\sigma\in(0,1)$ is a constant,  $K$ is a continuous function on $\Sn$,
\be\label{P sigma}
 P_\sigma=\frac{\Gamma(B+\frac{1}{2}+\sigma)}{\Gamma(B+\frac{1}{2}-\sigma)},\quad B=\sqrt{-\Delta_{g_{\Sn}}+\left(\frac{n-1}{2}\right)^2},
\ee
$c(n,\sigma)=\Gamma(\frac{n}{2}+\sigma)/\Gamma(\frac{n}{2}-\sigma)$, $\Gamma$ is the Gamma function and $\Delta_{g_{\Sn}}$ is the Laplace-Beltrami operator on $(\Sn, g_{\Sn})$. See \cite{Br}.
The operator $P_{\sigma}$ can be seen more concretely on $\R^n$ using stereographic projection.
The stereographic projection from $\Sn\backslash \{N\}$ to $\R^n$ is the inverse of
\[
F: \R^n\to \Sn\setminus\{N\}, \quad y\mapsto \left(\frac{2y}{1+|y|^2}, \frac{|y|^2-1}{|y|^2+1}\right),
\]
where $N$ is the north pole of $\Sn$. Then
\[
(P_\sigma(\phi))\circ F=  |J_F|^{-\frac{n+2\sigma}{2n}}(-\Delta)^\sigma(|J_F|^{\frac{n-2\sigma}{2n}}(\phi\circ F))\quad \mbox{for }\phi\in C^2(\Sn),
\]
where
\[
|J_F|=\left(\frac{2}{1+|y|^2}\right)^n,
\]
and $(-\Delta)^\sigma$ is the fractional Laplacian operator (see, e.g., page 117 of \cite{S}).
When $\sigma\in (0,1)$, Pavlov and Samko \cite{PS} showed that
\be\label{description of P sigma}
P_{\sigma}(v)(\xi)=P_{\sigma}(1)v(\xi)+c_{n,-\sigma}\int_{\Sn}\frac{v(\xi)-v(\zeta)}{|\xi-\zeta|^{n+2\sigma}}\,\ud vol_{g_{\Sn}}(\zeta)
\ee
for $v\in C^{2}(\Sn)$, where $c_{n,-\sigma}=\frac{2^{2\sigma}\sigma\Gamma(\frac{n+2\sigma}{2})}{\pi^{\frac{n}{2}}\Gamma(1-\sigma)}$ and $\int_{\Sn}$ is understood as $\lim\limits_{\va\to 0}\int_{|x-y|>\va}$.

 When $K=1$,
\eqref{main equ} is the
Euler-Lagrange equation for a functional associated to the following sharp Sobolev inequality on $\Sn$ (see \cite{Be})
\be\label{pe1}
\left(\dashint_{\mathbb{S}^n}|v|^{\frac{2n}{n-2\sigma}}\,\ud vol_{g_{\Sn}}\right)^{\frac{n-2\sigma}{n}}\leq \frac{\Gamma(\frac{n}{2}-\sigma)}{\Gamma(\frac{n}{2}+\sigma)}
\dashint_{\mathbb{S}^n}vP_{\sigma}(v)\,\ud vol_{g_{\Sn}}\quad \mbox{for }v\in H^{\sigma}(\mathbb{S}^n),
\ee
where $\dashint_{\mathbb{S}^n}=\frac{1}{|\mathbb{S}^n|}\int_{\mathbb{S}^n}$ and $H^{\sigma}(\mathbb{S}^n)$  is the closure of $C^{\infty}(\Sn)$ under the norm
\[
\begin{split}
\|u\|_{H^{\sigma}(\Sn)}:&=\int_{\Sn}vP_{\sigma}(v) \,\ud vol_{g_{\Sn}}\\
&=P_{\sigma}(1)\int_{\Sn}v^2\ud vol_{g_{\Sn}}+\frac{c_{n,-\sigma}}{2}\iint_{\Sn\times\Sn}\frac{(v(\xi)-v(\zeta))^2}{|\xi-\zeta|^{n+2\sigma}}\ud\xi\ud\zeta.
\end{split}
\]
The sharp Sobolev inequality on $\R^n$ is
\be\label{eq:sharp on rn}
\left(\int_{\mathbb{S}^n}|u|^{\frac{2n}{n-2\sigma}}\,\ud x\right)^{\frac{n-2\sigma}{n}}\leq \frac{\Gamma(\frac{n}{2}-\sigma)}{\omega_n^{\frac{2\sigma}{n}}\Gamma(\frac{n}{2}+\sigma)}
\|u\|^2_{\dot H^{\sigma}(\mathbb{R}^n)}\quad \mbox{for }u\in \dot H^{\sigma}(\mathbb{R}^n),
\ee
where $\omega_n$ is the volume of the unit sphere and $\dot H^{\sigma}(\mathbb{R}^n)$  is the closure of $C_c^{\infty}(\R^n)$ under the norm
\[
\begin{split}
\|u\|_{\dot H^{\sigma}(\R^n)}:&=\|(-\Delta)^{\sigma/2}u \|_{L^2(\R^n)}.
\end{split}
\]
The best constant and extremal functions of \eqref{eq:sharp on rn} were obtained in \cite{Lie83} and some classifications of solutions of \eqref{main equ} with $K\equiv 1$ can be found in \cite{CLO} and \cite{Li04}.

It is clear, by multiplying \eqref{main equ} by $v$, that a necessary condition for solving the problem is that
$K$ has to be positive somewhere. As in the classical case  \cite{KW}, the following Kazdan-Warner type condition
\be\label{1.3}
 \int_{\Sn}\langle \nabla_{g_{\Sn}} K, \nabla_{g_{\Sn}} \xi\rangle v^{\frac{2n}{n-2\sigma}}=0
\ee
gives another obstruction to solving \eqref{main equ}. The proof of \eqref{1.3} is given in \cite{JLX}.
Throughout the paper, we assume that $\sigma\in (0,1)$ and $n\ge 2$ without otherwise stated.
\begin{defn}
For $d>0$, we say that $K\in C(\Sn)$ has flatness order greater than $d$ at $\xi$ if, in some
local coordinate system $\{y_1,\cdots, y_n\}$ centered at $\xi$,
there exists a neighborhood $\mathscr{O}$ of $0$ such that $K(y)=K(0)+o(|y|^{d})$ in $\mathscr{O}$.
\end{defn}

\begin{thm}\label{K-M-E-S}
Let  $\sigma\in (0,1)$, and $K\in C^{1,1}(\Sn)$ be an antipodally symmetric function, i.e., $K(\xi)=K(-\xi)$ $\forall~\xi\in \Sn$, which is positive somewhere on $\Sn$.
If there exists a maximum point of $K$ at which $K$ has flatness order greater than $n-2\sigma$,
then \eqref{main equ} has at least one positive $C^{2}$ solution.
\end{thm}

%\begin{thm} \label{K-M-E-S} Let $0<K\in C^2(\Sn)$ a be antipodally symmetric function, i.e., $K(x)=K(-x)$ for any $x\in \Sn$. If
%$n<2+2\sigma$, then \eqref{1.1} has at least one positive $C^2$ solution.
%\end{thm}

For $2\leq n<2+2\sigma$, $K\in C^{1,1}(\Sn)$ has flatness order greater than $n-2\sigma$ at every maximum point.
When $\sigma=1$, the above theorem was proved by Escobar and Schoen \cite{ES} for $n\ge 3$. On $\mathbb{S}^2$, the existence of solutions of $-\Delta_{g_{\Sn}}v+1=Ke^{2v}$ for such $K$ was proved by Moser \cite{Ms}. Theorem \ref{K-M-E-S} was stated in \cite{JLX} and it is proved in Section \ref{proof of kmes}.

\begin{thm}\label{general exist}
 \label{main thm A} Suppose that $K\in C^{1,1}(\Sn)$ is a positive function satisfying that for every critical point
$\xi_0$ of $K$, in some geodesic normal coordinates $\{y_1, \cdots, y_n\}$ centered at $\xi_0$, there exist
some small neighborhood $\mathscr{O}$ of $0$ and positive constants $\beta=\beta(\xi_0)\in (n-2\sigma,n)$, $\gamma\in (n-2\sigma, \beta]$
such that $K\in C^{[\gamma],\gamma-[\gamma]}(\mathscr{O})$ (where $[\gamma]$ is the integer part of $\gamma$) and
\[
 K(y)=K(0)+\sum_{j=1}^{n}a_{j}|y_j|^{\beta}+R(y), \quad \mbox{in } \mathscr{O},
\]
where $a_j=a_j(\xi_0)\neq 0$, $\sum_{j=1}^n a_j\neq 0$,
$R(y)\in C^{[\beta]-1,1}(\mathscr{O})$ satisfies\\ $\sum_{s=0}^{[\beta]}|\nabla^sR(y)||y|^{-\beta+s} \to 0$ as $y\to 0$.
If
\[
\sum_{\xi\in \Sn\mbox{ such that }\nabla_{g_{\Sn}}K(\xi)=0,\  \sum_{j=1}^na_j(\xi)<0}(-1)^{i(\xi)}\neq (-1)^n,
\]
where
\[
 i(\xi)=\#\{a_j(\xi): \nabla_{g_{\Sn}}K(\xi)=0,a_j(\xi)<0,1\leq j\leq n\},
\]
then \eqref{main equ} has at least one $C^2$ positive solution. Moreover, there exists a positive constant $C$ depending only on $n, \sigma$ and $K$ such that for all positive $C^2$ solutions $v$ of \eqref{main equ},
\[
1/C\leq v\leq C\quad\mbox{and}\quad\|v\|_{C^2(\Sn)}\leq C.
\]
\end{thm}

For $n=3, \sigma=1$, the existence part of the above theorem was established by Bahri and Coron \cite{BC}, and the compactness part were given in Chang, Gursky and Yang \cite{CGY} and Schoen and Zhang \cite{SZ}. For $n\geq 4, \sigma=1$, the above theorem was proved by Li \cite{Li95}. The statement of Theorem \ref{general exist} and a proof of the compactness part were given in \cite{JLX}.  In Section
\ref{proof of general exist}, we prove the existence part of the theorem. The proof is based on a perturbation result, Theorem \ref{perturbation}, and an application of the Leray-Schauder degree. In Section \ref{improved inequality}, we prove a fractional Aubin inequality.

\bigskip

%\begin{thm}\label{thmpert} There exists a positive constant $\va=\va(n,\sigma)$ such that
%for any $0<K\in C^2(\Sn)$ satisfying the following non-degeneracy condition
%\[
%|\nabla_{g_{\Sn}}K(x)|+|\Delta_{g_{\Sn}}K(x)|\neq 0,\quad \mbox{for any }x\in \Sn,
%\]
%and $\max_{\Sn}K/\min_{\Sn} K\leq \va$, then  \eqref{main equ} has at least one positive $C^2$ solution provided
%\be\label{deg assum}
% \deg\left(\int_{\Sn} K\circ\varphi_{P,t}(x) x, B^{n+1}, 0\right)\neq 0,
%\ee
%for sufficiently large $t$.
%\end{thm}

%When $\sigma=1$, the above type result was initially established by Chang-Yang \cite{CY}. Recently, Chen-Xu \cite{CX} gave a new proof by flow apparach.
%The $\sigma=2, n/2$ cases were studied by Brendle \cite{B}, Djadli-Malchiodi-Ahmedou \cite{D1}, Wei-Xu \cite{WX}. The $1/2$ case is due to Chang-Xu-Yang \cite{CXY}.

\noindent\textbf{Acknowledgements:} T. Jin was supported in part by a University and Louis Bevier Dissertation Fellowship at Rutgers University and by Rutgers University School of Art and Science Excellence Fellowship. Y.Y. Li was supported in part by NSF (grant no. DMS-1065971 and DMS-1203961) and by Program for Changjiang Scholars and Innovative Research Team in University in China. J. Xiong was supported in part by
CSC project for visiting Rutgers University as a student and the First Class Postdoctoral Science Foundation of China (No. 2012M520002). He was very grateful to the Department of Mathematics at Rutgers University for the kind hospitality.

\section{Proof of Theorem \ref{K-M-E-S}}\label{proof of kmes}%%%%%%%%%%%%%%%%%%%%%%%%%%%%%%%%%%%%%%%%%%%%%%%%%%%%%%%%%%%%%%%%%%%%%%%%%%%%%%%%%%%%%%%%%%%%%%%%%%%%%%%%%%%%%%%%%%%%%%%%%%%%%%%%%%

Let $H^{\sigma}_{as}$ be the set of antipodally symmetric functions in $H^{\sigma}(\mathbb{S}^n)$, and let
\[
\lambda_{as}(K)=\inf_{v\in H^{\sigma}_{as}}\left\{\int_{\mathbb{S}^n}vP_{\sigma}(v): \int_{\mathbb{S}^n}K|v|^{\frac{2n}{n-2\sigma}}=1\right\}.
\]
We also denote $\omega_n$ as the volume of $\mathbb{S}^n$.
The proof of Theorem \ref{K-M-E-S} is divided into two steps.

\begin{prop}\label{less then exist1}
Let $K\in C^{1,1}(\mathbb{S}^n)$ be antipodally symmetric and positive somewhere. If
\begin{equation}\label{eq:less1}
\lambda_{as}(K)<\frac{P_{\sigma}(1)\omega_n^{\frac{2\sigma}{n}}2^{\frac{2\sigma}{n}}}{(\max_{\Sn}{K})^{\frac{n-2\sigma}{n}}},
\end{equation}
then there exists a positive and antipodally symmetric $C^{2}(\mathbb{S}^n)$ solution of \eqref{main equ}.
\end{prop}

\begin{prop}\label{less by test function1}
Let $K\in C^{1,1}(\mathbb{S}^n)$ be antipodally symmetric and positive somewhere. If there exists a maximum point of $K$ at which $K$ has flatness order greater than $n-2\sigma$, then
\begin{equation}\label{less than1}
\lambda_{as}(K)<\frac{P_{\sigma}(1)\omega_n^{\frac{2\sigma}{n}}2^{\frac{2\sigma}{n}}}{(\max_{\Sn}{K})^{\frac{n-2\sigma}{n}}}.
\end{equation}
\end{prop}

\begin{proof}[Proof of Theorem \ref{K-M-E-S}]
It follows from Proposition \ref{less then exist1} and Proposition \ref{less by test function1}.
\end{proof}

The proof of Proposition \ref{less then exist1} uses subcritical approximations. For $1<p<\frac{n+2\sigma}{n-2\sigma}$, we define
\[
\lambda_{as,p}(K)=\inf_{v\in H^{\sigma}_{as}}\left\{\int_{\mathbb{S}^n}vP_{\sigma}(v): \int_{\mathbb{S}^n}K|v|^{p+1}=1\right\}.
\]
We begin with a lemma
\begin{lem}\label{lem of existence subcritical}
Let $K\in C^{1,1}(\mathbb{S}^n)$ be antipodally symmetric and positive somewhere. Then $\lambda_{as,p}(K)$ is achieved by
a positive and antipodally symmetric $C^{2}(\Sn)$ function $v_{p}$, which satisfies
\begin{equation}\label{eq:subcritical1}
P_{\sigma}(v_p)=\lambda_{as,p}(K) K v_p^p\ \ \ \text{and}\ \ \ \int_{\mathbb{S}^n}Kv_p^{p+1}=1.
\end{equation}
\end{lem}

\begin{proof}
The existence of a nonnegative solution $v_p$ of \eqref{eq:subcritical1} follows from standard variational method and the inequality $\int_{\Sn}|v|P_{\sigma}(|v|)\leq \int_{\Sn} vP_{\sigma}(v)$ for all $v\in H^{\sigma}(\mathbb{S}^n)$.
%Then $v_p$ is positive everywhere by the Harnack inequality (see \cite{CaS}, \cite{TX} or \cite{JLX} in the extension point of view).
The regularity and positivity of $v_p$ follows from Proposition 2.4 and Theorem 2.1 in \cite{JLX}.
\end{proof}

%\marginpar{need to revise a little to make $K$ positive somewhere only}
\begin{proof}[Proof of Proposition \ref{less then exist1}]
First of all, we see that
\[
\limsup_{p\to\frac{n+2\sigma}{n-2\sigma}}\lambda_{as,p}(K)\leq \lda_{as}(K).
\]
Indeed, for any $\va>0$, there exists a nonnegative function $v\in H^{\sigma}_{as}$ such that
\[
\int_{\Sn}vP_{\sigma}(v)<\lda_{as}+\va \ \mbox{ and }\ \int_{\Sn}Kv^{\frac{2n}{n-2\sigma}}=1.
\]
Let $V_p:=\int_{\Sn}Kv^{p+1}$. Since $\lim_{p\to\frac{n+2\sigma}{n-2\sigma}}V_p=\int_{\Sn}Kv^{\frac{2n}{n-2\sigma}}=1$, we have, for $p$ closed to $\frac{n+2\sigma}{n-2\sigma}$,
\[
\lda_{as, p}(K)\leq\int_{\Sn}{\frac{v}{V_p^{1/(p+1)}}P_{\sigma}\left(\frac{v}{V_p^{1/(p+1)}}\right)}\leq \lda_{as}(K)+2\va.
\]
Hence, we may assume that there exists a sequence $\{p_i\}\to\frac{n+2\sigma}{n-2\sigma}$ such
that $\lambda_{as,p_i}(K)\to \lambda$ for some $\lambda\leq\lambda_{as}(K)$. Since $\{v_{i}\}$, which is a sequence of minimizers in Lemma \ref{lem of existence subcritical} for $p=p_i$, is bounded in
$H^{\sigma}(\mathbb{S}^n)$, then there exists $v\in H^{\sigma}(\Sn)$ such that $v_{i}\rightharpoonup v$ weakly in  $H^{\sigma}(\mathbb{S}^n)$ and $v$ is nonnegative. If $v\not\equiv 0$, it follows from \eqref{description of P sigma} that $v>0$ on $\Sn$, and we are done. Now we suppose that $v\equiv 0$. If $\{\|v_{i}\|_{L^{\infty}(\mathbb{S}^n)}\}$ is bounded, by the local estimates established in \cite{JLX} we have $\{\|v_{i}\|_{C^{2}(\mathbb{S}^n)}\}$ is bounded, too. Therefore, $v_{i}\to 0$ in $C^1(\Sn)$ which leads to $1=\int_{\mathbb{S}^n}K|v_{i}|^{p_i+1}\to 0$. This is a contradiction. Thus we may assume that $v_{i}(x_{i}):=\max_{\mathbb{S}^n}v_{i}\to\infty.$ Since $\Sn$ is compact, there exists a
subsequence of $\{x_{i}\}$, which will be still denoted as $\{x_{i}\}$, and $\bar x$ such that $x_{i}\to\bar x$.
Without loss of generality we assume that $\bar x$ is the south pole. Via the stereographic projection $F^{-1}$, \eqref{eq:subcritical1} becomes
\begin{equation}\label{eq:subcritical in plane1}
(-\Delta)^{\sigma}u_{i}(y)=\lambda_{as,p_i}(K)K\circ F(y)\left(\frac{2}{1+|y|^2}\right)^{\va_i}u_{i}^{p_i}(y),\ \ y\in\mathbb{R}^n
\end{equation}
where $v_{i}\circ F(y)=(\frac{1+|y|^2}{2})^{\frac{n-2\sigma}{2}}u_{i}(y)$ and $\va_i=\frac{n+2\sigma-p_i(n-2\sigma)}{2}$.
Thus for any $y\in\mathbb{R}^n$, $u_{i}(y)\leq 2^{\frac{n-2\sigma}{2}}u_{i}(y_{i})$
where $y_{i}:=F^{-1}(x_{i})\to 0.$ For simplicity, we denote $m_{i}:=u_{i}(y_{i})$. By our assumption on $v_{i}$ we have $m_{i}\to\infty.$
Define
\[
\tilde u_{i}(y)=(m_{i})^{-1}u_{i}\big((m_{i})^{\frac{1-p_i}{2\sigma}}y+y_{i}\big).
\]
From \eqref{eq:subcritical in plane1} we see that $\tilde u_{i}(y)$ satisfies, for any $y\in\mathbb{R}^n$,
\begin{equation}\label{eq:scaled subcritical in plane}
\begin{split}
(-\Delta)^{\sigma}\tilde u_{i}(y)=&\lambda_{as,p_i}(K)K\circ F(m_{i}^{\frac{1-p_i}{2\sigma}}y+y_{i})\\
&\quad\cdot\left(\frac{2}{1+|(m_{i})^{\frac{1-p_i}{2\sigma}}y+y_{i}|^2}\right)^{\va_i}\tilde u_{i}^{i}(y).
\end{split}
\end{equation}
Since $0<\tilde u_{i}\leq 2^{\frac{n-2\sigma}{2}}$, by the local estimates in \cite{JLX} $\{\tilde u_{i}\}$ is
bounded in $C^{2}_{loc}(\mathbb{R}^n)$.
%In particular $\{\tilde v_{p_i}\}$ is
%bounded in $C_{loc}^{2\sigma+2\va}$ for some $\va>0.$
Note that since $\{v_{i}\}$ is bounded in
$H^{\sigma}(\mathbb{S}^n)$, $\{\tilde u_{i}\}$ is bounded in $\dot H^{\sigma}(\R^n)$. Then there exists $u\in C^{2}(\mathbb{R}^n)\cap \dot H^{\sigma}(\R^n)$ such that, by passing to a subsequence,  $\tilde u_{i}\to u$ in
$C^{2}_{loc}(\mathbb{R}^n)$, $u(0)=1$, $\tilde u_{i}\rightharpoonup u$ weakly in $\dot H^{\sigma}(\R^n)$ and $u$ weakly satisfies
\begin{equation}\label{eq:critical in plane1}
(-\Delta)^{\sigma} u=\lambda K(\bar x)u^{\frac{n+2\sigma}{n-2\sigma}}.
\end{equation}
Hence $\lambda>0$, $K(\bar x)>0$, and the solutions of \eqref{eq:critical in plane1} are classified in \cite{CLO} and \cite{Li04} (see also Theorem 1.5 in \cite{JLX}).

%from here we can see that $\lda_G>0$ and $K(x-1)>0$, otherwise if one of them not, then u is a constant function or no solution which contradicts with H^{\sigma}. Hence at the beginning we can chose $\lda=\simsup \lda_i$

For $x\in\mathbb{S}^n$ and $r>0$, we denote $\mathcal{B}(x,r)$ be the geodesic ball centered at $x$ with
radius $r$ on $\mathbb{S}^n$, and for $y\in\mathbb{R}^n$ and $R>0$, we denote $B(y,R)$ be the
Euclidean ball in $\mathbb{R}^n$ of center $y$ and radius $R$.  For any $R>0$, let $\Omega_i:=F(B(y_{i},m_{i}^{\frac{1-p_i}{2\sigma}}R))$, we have
\begin{equation*}
\begin{split}
&\int_{\Omega_i}Kv_{i}^{p_i+1}=\int_{B(y_{i},m_{i}^{\frac{1-p_i}{2\sigma}}R)}K\circ F(y)\left(\frac{2}{1+|y|^2}\right)^{\va_i}u_{i}^{p_i+1}\\
&=\int_{B(0,R)}m_{i}^{\frac{\va_i}{2}}K\circ F((m_{i})^{\frac{1-p_i}{2\sigma}}y+y_{i})\left(\frac{2}{1+|(m_{i})^{\frac{1-p_i}{2\sigma}}y+y_{i}|^2}\right)^{\va_i}\tilde u_{i}^{p_i}(y)\\
&\ge K(\bar x)\int_{B(0,R)}u^{\frac{2n}{n-2\sigma}}+o(1)
\end{split}
\end{equation*}
as $\ p_i\to\frac{n+2\sigma}{n-2\sigma}$, where we used that $K$ is positive near $\bar x$, $\va_i\to 0$ and $\tilde u_{i}\to u$ in $C^{2}_{loc}(\mathbb{R}^n)$. %Thus for any small $\delta>0$,
%\[
%\int_{\mathcal{B}(\bar x,\delta)}Kv_{p_i}^{p_i+1}\ge K(\bar x)\int_{\mathbb{R}^n}u^{\frac{2n}{n-2\sigma}}+o(1).
%\]
Since $K$ and $v_{i}$ are antipodally symmetric,
%\[
%\int_{\mathcal{B}(-x_1,\delta)}Kv_{p_i}^{p_i+1}=\int_{\mathcal{B}(x_1,\delta)}Kv_{p_i}^{p_i+1}\ge K(x_1)\int_{\mathbb%{R}^n}u^{\frac{2n}{n-2\sigma}}+o(1).
%\]
we have, by taking $\delta$ small and $R$ sufficiently large,
\begin{equation}\label{norm less1}
\begin{split}
1=\int_{\mathbb{S}^n}Kv_{i}^{p_i+1}&\ge 2\int_{\mathcal{B}(x_1,\delta)}Kv_{i}^{p_i+1}+\int_{\{K<0\}}Kv_{i}^{p_i+1}\\
&=2K(\bar x)\int_{\mathbb{R}^n}u^{\frac{2n}{n-2\sigma}}+\int_{\{K<0\}}Kv_{i}^{p_i+1}.
\end{split}
\end{equation}
We claim that
\[
\int_{\{K<0\}}Kv_{i}^{p_i+1}\to 0\quad\mbox{as}\quad p_i\to\frac{n+2\sigma}{n-2\sigma}.
\]
Indeed, for any $\va>0$, it is not difficult to show, by blow up analysis, that $\|v_{i}\|_{L^{\infty}(\Omega_{\va/4})}\leq C(\va)$ where $\Omega_{\va}:=\{x\in \Sn: K(x)<-\va\}$ and $C(\va)$ is independent of $p_i$. By the local estimates established in \cite{JLX}, we have $\|v_{i}\|_{C^{2}(\Omega_{\va/2})}\leq C(\va)$ and hence $v_{i}\to 0$ in $C^1(\overline\Omega_{\va})$ (recall that we assumed that $v_{i}\rightharpoonup 0$ weakly in $H^{\sigma}(\Sn)$). Thus when $p_i$ is sufficiently close to $\frac{n+2\sigma}{n-2\sigma}$,
\[\int_{\Omega_{\va}}|K|v_{i}^{p_i+1}<\va.
\]
On the other hand, by H\"older inequality and Sobolev inequality,
\[
\int_{-\va\le K<0}|K|v_{i}^{p_i+1}<C(n,\sigma)\va\|v_{i}\|_{L^{\frac{2n}{n-2\sigma}}}^{p_i+1}\leq C(n,\sigma,\lambda_{as})\va,
\]
which finishes the proof of our claim. Thus, \eqref{norm less1} leads to
\be\label{norm less2}
1\ge 2K(\bar x)\int_{\mathbb{R}^n}u^{\frac{2n}{n-2\sigma}}+o(1).
\ee
By the sharp Sobolev inequality \eqref{pe1}, \eqref{eq:critical in plane1} and \eqref{norm less2}, we have
\begin{equation*}
\begin{split}
P_{\sigma}(1)\omega_n^{\frac{2\sigma}{n}}\leq \frac{\int_{\mathbb{R}^n}u(-\Delta)^{\sigma}u}{\left(\int_{\mathbb{R}^n}
u^{\frac{2n}{n-2\sigma}}\right)^{\frac{n-2\sigma}{n}}}&=\lambda K(\bar x)\left(\int_{\mathbb{R}^n}u^{\frac{2n}{n-2\sigma}}\right)^{\frac{2\sigma}{n}}
\\
&\leq \lambda_{as}(K)K(\bar x)(2K(\bar x))^{-2\sigma/n}\\
&\leq \lambda_{as}(K)2^{-2\sigma/n}(\max_{\Sn} K)^{1-2\sigma/n},
\end{split}
\end{equation*}
which contradicts with \eqref{eq:less1}.
\end{proof}

Next we shall prove Proposition \ref{less by test function1} using some test functions, which are inspired by \cite{Hebey, Ro}.

\begin{proof}[Proof of Proposition \ref{less by test function1}]
Let $\xi_1$ be a maximum point of $K$ at which $K$ has flatness order greater than $n-2\sigma$. Suppose $\xi_2$ is the antipodal point of $\xi_1$.
For $\beta>1$ and $i=1,2$ we define
\be\label{eq:bubble}
v_{i,\beta}(x)=\left(\frac{\sqrt{\beta^2-1}}{\beta-\cos r_i}\right)^{\frac{n-2\sigma}{2}},
\ee
where $r_i=d(x,\xi_i)$ is the geodesic distance between $x$ and $\xi_i$ on the sphere. It is clear that
\[
P_{\sigma}(v_{i,\beta})=P_{\sigma}(1)v_{i,\beta}^{\frac{n+2\sigma}{n-2\sigma}}\quad\mbox{and}\quad\int_{\mathbb{S}^n}v_{i,\beta}^{\frac{2n}{n-2\beta}}=\omega_n.
\]
Let
\[
v_{\beta}=v_{1,\beta}+v_{2,\beta},
\]
which is antipodally symmetric. Then
\begin{equation*}
\begin{split}
\int_{\mathbb{S}^n}v_{\beta}P_{\sigma}(v_{\beta})&=P_{\sigma}(1)\int_{\Sn}\dsum_{i=1}^2 v_{i,\beta}^{\frac{n+2\sigma}{n-2\sigma}}\dsum_{j=1}^2 v_{j,\beta}\\
&=P_{\sigma}(1)\int_{\Sn}\dsum_{i=1}^2 v_{i,\beta}^{\frac{2n}{n-2\sigma}}+2 v_{1,\beta}^{\frac{n+2\sigma}{n-2\sigma}}v_{2,\beta}\\
&=P_{\sigma}(1)2\omega_n \left(1+\omega_n^{-1}\int_{\Sn} v_{1,\beta}^{\frac{n+2\sigma}{n-2\sigma}}v_{2,\beta}\right).
\end{split}
\end{equation*}
By direct computations with change of variables, we have
\[
\int_{\Sn}v_{1,\beta}^{\frac{n+2\sigma}{n-2\sigma}}v_{2,\beta}=A(\beta-1)^{\frac{n-2\sigma}{2}}+o\big((\beta-1)^{\frac{n-2\sigma}{2}}\big)
\]
for $\beta$ close to $1$, where
\[
A=2^{-\frac{n-2\sigma}{2}}\omega_{n-1}\int_0^{+\infty}\frac{2^nr^{n-1}}{(1+r^2)^{\frac{n+2\sigma}{2}}}dr>0.
\]
Choose a sufficiently small neighborhood $V_1$ of $\xi_1$ and let $V_2=\{x\in\Sn: -x\in V_1\}$. Then $K$ is positive in $V_1\cup V_2$ and
\begin{equation*}
\begin{split}
\int_{\mathbb{S}^n}K v_{\beta}^{\frac{2n}{n-2\sigma}}%&=\int_{\mathbb{S}^n}f\left(\sum\limits_{i=1}^m u_{i,\beta}\right)^{\frac{2n}{n-2\sigma}}\\
&=\int_{\cup V_i}K\left(v_{1,\beta}+v_{2,\beta}\right)^{\frac{2n}{n-2\sigma}}+\int_{\mathbb{S}^n\backslash \cup V_i}Kv_{\beta}^{\frac{2n}{n-2\sigma}}\\
&=2\int_{V_1}K\left(v_{1,\beta}+v_{2,\beta}\right)^{\frac{2n}{n-2\sigma}}+\int_{\mathbb{S}^n\backslash \cup V_i}Kv_{\beta}^{\frac{2n}{n-2\sigma}}\\
&\geq 2\int_{V_1}K\left(v_{1,\beta}^{\frac{2n}{n-2\sigma}}+\frac{2n}{n-2\sigma}v_{1,\beta}^{\frac{n+2\sigma}{n-2\sigma}} v_{2,\beta}\right)+\int_{\mathbb{S}^n\backslash \cup V_i}Kv_{\beta}^{\frac{2n}{n-2\sigma}}.
%&=2\int_{V_1}Kv_{1,\beta}^{\frac{2n}{n-2\sigma}}+\frac{4n}{n-2\sigma}\int_{V_1}Kv_{1,\beta}^{\frac{n+2\sigma}{n-2\sigma}}v_{2,\beta}+\int_{\mathbb{S}^n\backslash \cup V_i}Kv_{\beta}^{\frac{2n}{n-2\sigma}}.
\end{split}
\end{equation*}
Since $K(x)$ is flat of order $n-2\sigma$ at $\xi_1$, we have in $V_1$ that,
\[
K(x)=K(\xi_1)+o(1)|x-\xi_1|^{n-2\sigma}.
\]
Thus
\begin{equation*}
\begin{split}
\int_{\mathbb{S}^n}Kv_{\beta}^{\frac{2n}{n-2\sigma}}&\geq 2K(\xi_1)\int_{\Sn}v_{1,\beta}^{\frac{2n}{n-2\sigma}}+\frac{4nA}{n-2\sigma}K(\xi_1)(\beta-1)^{\frac{n-2\sigma}{2}}+o\big((\beta-1)^{\frac{n-2\sigma}{2}}\big)\\
&=2K(\xi_1)\omega_n\left(1+\frac{2nA}{n-2\sigma}\omega_n^{-1}(\beta-1)^{\frac{n-2\sigma}{2}}+o\big((\beta-1)^{\frac{n-2\sigma}{2}}\big)\right)
\end{split}
\end{equation*}
for $\beta$ close to $1$. Hence
\[
\frac{\int_{\mathbb{S}^n}v_{\beta}P_{\sigma}(v_{\beta})}{\left(\int_{\mathbb{S}^n}Kv_{\beta}^{\frac{2n}{n-2\sigma}}\right)^{\frac{n-2\sigma}{n}}}\leq
\frac{P_{\sigma}(1)\omega_n^{\frac{2\sigma}{n}}2^{\frac{2\sigma}{n}}}{K(\xi_1)^{\frac{n-2\sigma}{n}}}
\left(1-\frac{A}{\omega_n}(\beta-1)^{\frac{n-2\sigma}{2}}+o\big((\beta-1)^{\frac{n-2\sigma}{2}}\big)\right),
\]
which implies \eqref{less than1} holds.
\end{proof}

%In this section, we are going to prove the following theorem which implies Theorem \ref{K-M-E-S}.
Theorem \ref{K-M-E-S} can be extended to positive functions $K$ which are invariant under some isometry group acting without fixed points (see \cite{Hebey, Ro}). Denote $Isom(\mathbb{S}^n)$ as the isometry group of the standard sphere $(\mathbb{S}^n, g_{\Sn})$.
Let $G$ be a subgroup of $Isom(\mathbb{S}^n)$. We say that $G$ acts without fixed points if for
each $x\in\mathbb{S}^n$, the orbit $O_{G}(x):=\{g(x)| g\in G\}$ has at least two elements.
We denote $|O_{G}(x)|$ be the number of elements in $O_{G}(x)$. A function $K$ is called $G$-invariant
if $K\circ g\equiv K$ for all $g\in G$.

\begin{thm}\label{invariance nirenberg}
Let $G$ be a finite subgroup of $Isom(\mathbb{S}^n)$ and act without fixed points.
Let $K\in C^{1,1}(\mathbb{S}^n)$ be a positive and G-invariant function.
If there exists $\xi_0\in\mathbb{S}^n$ such that  $K$ has flatness order greater than $n-2\sigma$ at $\xi_0$, and for any $x\in\mathbb{S}^n$
\begin{equation}\label{eq:condition 1}
\frac{K(\xi_0)}{|O_{G}(\xi_0)|^{\frac{2\sigma}{n-2\sigma}}}\geq \frac{K(x)}{|O_{G}(x)|^{\frac{2\sigma}{n-2\sigma}}},
\end{equation}
then \eqref{main equ} possesses a positive and G-invariant $C^{2}(\mathbb{S}^n)$ solution.
\end{thm}

Let $H^{\sigma}_{G}$ be the set of G-invariant functions in $H^{\sigma}(\mathbb{S}^n)$. Let
\[
\lambda_{G}(K)=\inf_{v\in H^{\sigma}_{G}}\left\{\int_{\mathbb{S}^n}vP_{\sigma}(v): \int_{\mathbb{S}^n}K|v|^{\frac{2n}{n-2\sigma}}=1\right\}.
\]
Similar to Theorem \ref{K-M-E-S}, the proof of Theorem \ref{invariance nirenberg} is again divided into two steps.

\begin{prop}\label{less then exist}
Let $G$ be a finite subgroup of $Isom(\mathbb{S}^n)$. Let $K\in C^{1,1}(\mathbb{S}^n)$ be a positive and G-invariant function. If for all $x\in\mathbb{S}^n$,
\begin{equation}\label{eq:less}
\lambda_{G}(K)<\frac{P_{\sigma}(1)\omega_n^{\frac{2\sigma}{n}}|O_{G}(x)|^{\frac{2\sigma}{n}}}{K(x)^{\frac{n-2\sigma}{n}}},
\end{equation}
then there exists a positive G-invariant $C^{2}(\mathbb{S}^n)$ solution of \eqref{main equ}.
\end{prop}

\begin{prop}\label{less by test function}
Let $G$ be a finite subgroup of $Isom(\mathbb{S}^n)$ and act without fixed points.
Let $K\in C^{1,1}(\mathbb{S}^n)$ be a positive and G-invariant function. If $K$ has flatness order greater than $n-2\sigma$ at $\xi_1$ for some $\xi_1\in\mathbb{S}^n$, then
\begin{equation}\label{less than}
\lambda_{G}(K)<\frac{P_{\sigma}(1)\omega_n^{\frac{2\sigma}{n}}|O_{G}(\xi_1)|^{\frac{2\sigma}{n}}}{K(\xi_1)^{\frac{n-2\sigma}{n}}}.
\end{equation}
\end{prop}

Theorem \ref{invariance nirenberg} follows from Proposition \ref{less then exist} and Proposition \ref{less by test function} immediately. The proof of Proposition \ref{less then exist} uses subcritical approximations and blow up analysis, which is similar to that of Proposition \ref{less then exist1}. Proposition \ref{less by test function} can be verified by the following G-invariant test function
\[
v_{\beta}=\sum\limits_{i=1}^m v_{i,\beta},
\]
where $m=|O_G(\xi_1)|$, $O_{G}(\xi_1)=\{\xi_1,\dots,\xi_m\}$, $\xi_i=g_i(\xi_1)$ for some $g_i\in G$, $g_1=Id$, $v_{j,\beta}:=v_{1,\beta}\circ g_i^{-1}$ and $v_{1,\beta}$ is as in \eqref{eq:bubble}. We omit the detailed proofs of Propositions \ref{less then exist} and \ref{less by test function}, and leave them to the readers.

\section{Proof of Theorem \ref{general exist}}
\label{proof of general exist}

In this section, we first establish a perturbation result, see Theorem \ref{perturbation}. The method we shall use is smilar to that in \cite{Li95},
but we have to set up a framework to fit the fractional situation. Perturbation results in the classical Nirenberg problem
were obtained in \cite{CY87}, \cite{CY}, \cite{Li95} and many others.

For a conformal transformation $\varphi:\mathbb{S}^n\to \mathbb{S}^n$, we set
\[
 T_{\varphi}v=v\circ\varphi|\det d\varphi|^{(n-2\sigma)/2n},
\]
where $|\det d\varphi|$ denotes the Jacobian of $\varphi$ satisfying
\[
 \varphi^*g_{\Sn}=|\det d\varphi|^{2/n}g_{\Sn}.
\]

\begin{lem}\label{lempe1}
For any conformal transform $\varphi$ on $\Sn$ we have
\[
 \int_{\mathbb{S}^n}T_{\varphi}vP_\sigma(T_{\varphi}v)\,\ud vol_{g_{\Sn}}=\int_{\mathbb{S}^n} vP_\sigma(v)\,\ud vol_{g_{\Sn}}
\]
and
\[
 \int_{\mathbb{S}^n}|T_{\varphi}v|^{\frac{2n}{n-2\sigma}}\,\ud vol_{g_{\Sn}}=\int_{\mathbb{S}^n}|v|^{\frac{2n}{n-2\sigma}}\,\ud vol_{g_{\Sn}}
\]for all $v\in H^\sigma(\mathbb{S}^n)$.
\end{lem}

\begin{proof} We only prove the first equality.
 Recall that $P_\sigma=P_\sigma^{g_{\Sn}}$. By the conformal invariance of $P_\sigma^g$,
\[
 \begin{split}
  \int_{\mathbb{S}^n} vP_\sigma(v)\,\ud vol_{g_{\Sn}}&=\int_{\mathbb{S}^n} v \circ \varphi P^{\varphi^*g_{\Sn}}_\sigma(v\circ \varphi)\,\ud vol_{\varphi^*g_{\Sn}}\\&
=\int_{\mathbb{S}^n} v \circ \varphi P^{|\det \varphi|^{2/n}g_{\Sn}}_\sigma(v\circ \varphi)|\det \varphi|\,\ud vol_{g_{\Sn}}\\&
=\int_{\mathbb{S}^n} v \circ \varphi |\det \varphi|^{-\frac{n+2\sigma}{2n}} P^{g_{\Sn}}_\sigma(v\circ \varphi |\det \varphi|^{\frac{n-2\sigma}{2n}})|\det \varphi|\,\ud vol_{g_{\Sn}}\\
&= \int_{\mathbb{S}^n}T_{\varphi}vP_\sigma(T_{\varphi}v)\,\ud vol_{g_{\Sn}}.
 \end{split}
\]
\end{proof}
For $P\in \mathbb{S}^n$, $1\leq t<\infty$, we recall a conformal transform (see, e.g., \cite{CY87})
\be\label{phi}
 \varphi_{P,t}:\mathbb{S}^n\to \mathbb{S}^n,\quad y\mapsto t y,
\ee
where $y$ is the stereographic projection coordinates of points on $\mathbb{S}^n$ while the stereographic projection is performed with $P$
as the north pole to the equatorial plane of $\mathbb{S}^n$. The totality of such a set of conformal transforms is diffeomorphic to the unit ball
$B^{n+1}$ in $\R^{n+1}$, with the identity transformation identified with the origin in $B^{n+1}$ and
\[\varphi_{P,t}\leftrightarrow ((t-1)/t)P=:p\in B^{n+1}\]
in general. We denote $\varphi_p=\varphi_{P,t}$.
Let
\[
 \M=\left\{v \in H^{\sigma}(\mathbb{S}^n):\dashint_{\mathbb{S}^n}|v|^{\frac{2n}{n-2\sigma}}\,\ud vol_{g_{\Sn}}=1\right\},
\]
\[
 \M_0=\left\{v\in \M:\dashint_{\mathbb{S}^n} x|v|^{\frac{2n}{n-2\sigma}}\,\ud vol_{g_{\Sn}}=0\right\}.
\]
Define
\[
 \varpi:\M_0\times B^{n+1}\to \M
\]
by
\[
 v=\varpi(w,p)=T^{-1}_{\varphi_p} w,\quad w\in \M_0.
\]
%where $\varphi_p$ is defined at the beginning of the paper.

\begin{lem} \label{lem c2 diff}
$\varpi: \M_0\times B_1\to \M$ is a $C^2$ diffeomorphism.
\end{lem}

\begin{proof}
The proof is the same as that of Lemma 5.4 in \cite{Li95}.
\end{proof}

Consider the following functional on $\M$
\[
 E_{K}(v)=\frac{\dashint_{\mathbb{S}^n}vP_\sigma(v)\,\ud vol_{g_{\Sn}}}{\left(\dashint_{\mathbb{S}^n}K|v|^{2n/(n-2\sigma)}\,\ud vol_{g_{\Sn}}\right)^{(n-2\sigma)/n}},
\]
where $K>0$.
By \eqref{pe1} we see that the functional $E_K$ has a lower bound over $\M$ provided $K$ is positive and bounded.
By Lemma \ref{lempe1}, $E_K$ is in fact defined on $\M_0$ when $K$ is a constant.
Due to the classification of extremas of \eqref{pe1},
$\min_{w\in \M_0}E_1=P_{\sigma}(1)$ is achieved only by $-1$ and $1$.

Note that both $\M$ and $\M_0$ are $C^2$ surfaces in the Hilbert space $H^\sigma(\Sn)$.

\begin{lem}\label{lempe2}
Let $T_1\M$ denote the tangent space of $\M$ at $v=1$, then we have
\[
\begin{split}
 T_1\M&=\left\{\phi: \int_{\mathbb{S}^n}\phi=0\right\}\\&
=\mathrm{span}\{\mbox{spherical harmonics of degree }\geq 1\}.
\end{split}
\]
\end{lem}

\begin{lem}\label{lempe3}
 Let $T_1\M_0$ denote the tangent space of $\M_0$ at $v=1$, then we have
\[
 T_1\M_0=\mathrm{span}\{\mbox{spherical harmonics of degree } \geq 2\}.
\]
\end{lem}

The above two Lemmas follow from direct computations and some elementary properties of spherical harmonics.

The following lemma can be proved by the Implicit Function Theorem (see Lemma 6.4 in \cite{Li95}).

\begin{lem}\label{lempe4}
For $\tilde{w}\in T_1\M_0$, $\tilde{w}$ close to $0$, there exist $\mu(\tilde{w})\in \R$, $\eta(\tilde{w})\in \R^{n+1}$ being $C^2$ functions such that
\be\label{pe2}
 \dashint_{\mathbb{S}^n}|1+\tilde{w}+\mu+\eta\cdot x|^{2n/(n-2\sigma)}=1
\ee
and
\be\label{pe3}
 \int_{\Sn}|1+\tilde{w}+\mu+\eta\cdot x|^{2n/(n-2\sigma)}x=0.
\ee
Furthermore, $\mu(0)=0, \eta(0)=0, D\mu(0)=0$ and $D\eta(0)=0$.
\end{lem}

Let us use $\tilde{w}\in T_1\M_0$ as local coordinates of $w\in \M_0$ near $w=1$, and $\tilde{w}=0$ corresponds to $w=1$.

Let
\[
 \tilde{E}(\tilde{w})=E_1(w)=\dashint_{\mathbb{S}^n}wP_\sigma(w),
\]
where $\tilde{w}\in T_1\M_0$ and $w=1+\tilde{w}+\mu(\tilde{w})+\eta(\tilde{w})\cdot x$ as in Lemma \ref{lempe4}.
It is well-known (see, e.g.  \cite{Mo}) that $P_\sigma$ has eigenfunctions the spherical harmonics and
eigenvalues
\[
 \lda_k=\frac{\Gamma(k+\frac{n}{2}+\sigma)}{\Gamma(k+\frac{n}{2}-\sigma)}, \quad k\geq 0,
\]
with multiplicity $(2k+n-1)(k+n-2)!/(n-1)!k!$. Note that $\lda_0=P_{\sigma}(1)$.
Since $P_\sigma$ is a linear operator, it follows from Lemma \ref{lempe2}, Lemma \ref{lempe3} and Lemma \ref{lempe4} that
\[
\tilde{E}(\tilde{w})=P_{\sigma}(1)(1+2\mu(\tilde{w}))+\dashint_{\mathbb{S}^n}\tilde{w}P_\sigma(\tilde{w})+o(\|\tilde{w}\|^2_{H^\sigma(\mathbb{S}^n)}).
\]
By \eqref{pe2}, it follows that
\[
\begin{split}
\mu(\tilde{w})&=\frac12 D^2\mu(0)(\tilde{w},\tilde{w})+o(\|\tilde{w}\|^2_{H^\sigma(\mathbb{S}^n)})\\&
=-\frac12\cdot \frac{n+2\sigma}{n-2\sigma}\dashint_{\Sn}\tilde{w}^2+o(\|\tilde{w}\|^2_{H^\sigma(\mathbb{S}^n)}).
\end{split}
\]
Note that $\lda_1=\frac{n+2\sigma}{n-2\sigma}P_{\sigma}(1)$. Therefore, we have
\be\label{pe4}
\tilde{E}(\tilde{w})=P_{\sigma}(1)+\dashint_{\mathbb{S}^n}(\tilde{w}P_\sigma(\tilde{w})-\lda_1\tilde{w}^2)
+o(\|\tilde{w}\|^2_{H^\sigma(\mathbb{S}^n)}).
\ee
Set $Q(\tilde{w}):=\dashint_{\mathbb{S}^n}(\tilde{w}P_\sigma(\tilde{w})-\lda_1\tilde{w}^2)$.
By Lemma \ref{lempe3}, we see that for any $\tilde{w},\tilde{v}\in T_1\M_0$
\be \label{pe5}
\begin{split}
D^2Q(\tilde{w})(\tilde{v},\tilde{v})&=2\dashint_{\mathbb{S}^n}\tilde{v}P_\sigma(\tilde{v})-
\lda_1\tilde{v}\tilde{v}\\&
\geq 2(1-\frac{\lda_1}{\lda_2})\|\tilde{v}\|^2_{H^\sigma(\Sn)},
\end{split}
\ee
which means the quadratic form $Q(\tilde{w})$ is positive definite in $T_1\M_0$. Moreover, it follows from \eqref{pe1} that
for some $\va_1=\va_1(n,\sigma)>0$,
\be\label{pe6}
 \|E_K|_{\M_0}-E_1|_{\M_0}\|_{C^2(B_{\va_1}(1))}\leq O(\va),
\ee
provided $\|K-1\|_{L^\infty(\mathbb{S}^n)}\leq \va$. Here $B_{\va_1}(1)$ denotes the ball in $\M_0$ of
radius $\va_1$ centered at $1$.

It is elementary to compute that for any $\tilde{w}\in T_1\M_0$, we have, for any constant $c$, that
\[
 \langle DE_K|_{\M_0}(1),\tilde{w}\rangle=-2P_{\sigma}(1)\left(\dashint_{\mathbb{S}^n}K\right)^{(2\sigma-2n)/n}\dashint_{\mathbb{S}^n}(K-c)\tilde{w}.
\]
It follows that
\[
\begin{split}
 | \langle DE_K|_{\M_0}(1),\tilde{w}\rangle|&\leq C\|K-c\|_{L^{2n/(n+2\sigma)}}\|\tilde{w}\|_{L^{2n/(n-2\sigma)}}\\&
 \leq C\|K-c\|_{L^{2n/(n+2\sigma)}}\|\tilde{w}\|_{H^{\sigma}}.
 \end{split}
\]
Therefore,
\be\label{pe7}
\| DE_K|_{\M_0}(1)\|\leq C\|K-c\|_{L^{2n/(n+2\sigma)}}.
\ee

\begin{lem}\label{lempe5}
Let $K\in C^1(\Sn)$. There exist $\va_2=\va_2(n,\sigma)>0$, $\va_3=\va_3(n,\sigma)>0,$ such that,  if
$\|K-1\|_{L^\infty(\mathbb{S}^n)}\leq \va \leq \va_2$,
\[
 \min_{w\in \M_0,\ \|w-1\|_{H^\sigma}\leq \va_3} E_K(w)
\]
has a unique minimizer $w_K$. Furthermore, $D^2E_K|_{\M_0}(w_K)$ is positive definite and
\be\label{pe8}
w_K>0\quad \mbox{on }\mathbb{S}^n,
\ee
\be\label{pe9}
\|w_K-1\|_{H^\sigma}\leq C(n,\sigma)\inf_{c\in \R}\|K-c\|_{L^{2n/(n+2\sigma)}},
\ee
\be\label{pe10}
\|w_K-1\|_{L^{\infty}}+\|P_\sigma(w_K-1)\|_{L^{\infty}}\leq o_{\va}(1),
\ee
where $o_{\va}(1)$ denotes some quantity depending only on $n,\sigma$ which tends to $0$ as $\va\to 0$.
If $\sigma\geq \frac{1}{2}$, then there exists $C(n,\sigma,\va_2)>0$ such that
\be\label{gradient estimate for wK}
\|\nabla w_K\|_{L^2}\leq C(n,\sigma,\va_2) \inf_{\tilde c\in [1/2,2]}\|K-\tilde c\|_{L^2}.
\ee
\end{lem}

\begin{proof}
It follows from \eqref{pe4}, \eqref{pe5} and \eqref{pe6} that the minimizing problem has a unique minimizer $w_K$
and $D^2E_K|_{\M_0}(w_K)$ is positive definite. \eqref{pe9} follows from \eqref{pe5}, \eqref{pe6}, \eqref{pe7} and some standard functional analysis arguments.

Since $w_K$ is a constrained local minimum, $w_K$ satisfies the Euler-Lagrange equation for some
Lagrange multiplier $\Lda_K\in \R^{n+1}$:
\be\label{pe11}
P_\sigma(w_K)=(\lda_KK-\Lda_K\cdot x)|w_K|^{4\sigma/(n-2\sigma)}w_K\quad \mbox{on } \mathbb{S}^n,
\ee
where
\[
 \lda_K=\frac{\dashint_{\mathbb{S}^n}w_KP_\sigma(w_K)}{
 %\left(
 \dashint_{\mathbb{S}^n}K|w_K|^{\frac{2n}{n-2\sigma}}
 %\right)^{(n-2\sigma)/n}
 }.
\]
It is clear that $|\lda_K-c(n,\sigma)|=O(\va)$ (recall that $c(n,\sigma)$ is defined in \eqref{main equ}, which equals to $P_{\sigma}(1)$). Since $P_\sigma$ is a self-adjoint operator and $P_\sigma(\Lda_K\cdot x)=\lda_1(n,\sigma)\Lda_K\cdot x$,
multiplying \eqref{pe11} by $\Lda_K\cdot x$ and integrating over both sides we have
\be\label{estimate for Lda}
\begin{split}
 &\lda_1(n,\sigma)\int_{\mathbb{S}^n}w_K\Lda_K\cdot x\\&=\lda_K\int_{\mathbb{S}^n}K|w_K|^{4\sigma/(n-2\sigma)}w_K\Lda_K\cdot x-
\int_{\mathbb{S}^n}(\Lda_K\cdot x)^2|w_K|^{4\sigma/(n-2\sigma)}w_K.
\end{split}
\ee
Making use of the fact that $\|K-1\|_{L^\infty}\leq \va$ and $\|w_K-1\|_{H^\sigma}\leq O(\va)$, we conclude that
\be\label{eq:lda bound}
 |\Lda_K|=O(\va).
\ee
Set $w_K=w_K^+-w_K^-$. Note that $\int_{\mathbb{S}^n}|w_K^-|^{2n/(n-2\sigma)}\leq\int_{\mathbb{S}^n}|w_K-1|^{2n/(n-2\sigma)}\leq o_\va(1)$. On the other hand, we have, by multiplying \eqref{pe11} by
$-w_K^-$,
\[
\begin{split}
C\int_{\Sn}(w_K^-)^{\frac{2n}{n-2\sigma}}
&\geq \int_{\mathbb{S}^n}w_K^-P_\sigma(-w_K)\\
&\geq\int_{\mathbb{S}^n}w_K^-P_\sigma(w_K^-)\\%\leq o_\va(1)\int_{\mathbb{S}^n}w_K^-P_\sigma(w_K^-),
&\geq c(n,\sigma)\left(\int_{\Sn}(w_K^-)^{\frac{2n}{n-2\sigma}}\right)^{\frac{n-2\sigma}{n}},
\end{split}
\]where we used \eqref{pe1}.
%\[
% \int_{\mathbb{S}^n}w_K^-P_\sigma(w_K^-)\leq  \int_{\mathbb{S}^n}w_K^-P_\sigma(-w_K).
%\]
Therefore, we conclude that $w_K^-=0$. Then \eqref{pe8} follows from \eqref{description of P sigma} and \eqref{pe11}.

It follows from \eqref{pe9}, \eqref{pe11}, \eqref{eq:lda bound}, Lemma 2.2 in \cite{JLX} and Proposition 2.4 in \cite{JLX} that
\be\label{equation:wK bound}
\|w_K-1\|_{L^{\infty}} \leq o_{\va}(1),
\ee
%$w_K\leq C(n,\sigma)$.
%Once we know $w_K$ is positive and \eqref{pe9}, by a blow up argument it is easy to show that $w_K\leq C$.
which, together with \eqref{pe11}, leads to
\[
 \|P_\sigma(w_K-1)\|_{L^{\infty}} \leq o_{\va}(1).
 \]
 Then \eqref{pe10} follows immediately.
% suppose f is the weighted projection, we can use Moser iteration to get L^{\infty} for f locally, then w is locally bounded, i.e. globally bounded, then go back to f, i.e. f is global bounded with rate at infinity, then apply silvetre's proposition2.9, f's C^{2\sigma} norm is bounded globally. then use silvestre's proposition 2.8 to iterate to get f is C^1 bounded hy H^{\sigma} norm. then we go back to w which is bounded in C^{1}.

By \eqref{equation:wK bound} and \eqref{pe9}, we can see that for any $\tilde c\in [1/2,2]$,
\[
\begin{split}
|\lda_K-P_{\sigma}(1)/\tilde c|&\leq C(n,\sigma,\va_2)(\|w_k-1\|_{H^{\sigma}}+\|w_k-1\|_{L^1}+\|K-\tilde c\|_{L^1})\\
&\leq C(n,\sigma,\va_2)\|K-\tilde c\|_{L^{2n/(n+2\sigma)}}.
\end{split}
\]
From \eqref{estimate for Lda}, \eqref{pe8}, \eqref{equation:wK bound} and the fact that $w_K\in \M_0$, we have that
\[
|\Lda_K|^2\leq C(n,\sigma)|\Lda_K|\left(\int_{\Sn}|K-\tilde c|+\int_{\Sn}|w_K-1|\right).
\]
By \eqref{pe9} and H\"older inequalities we have
\[
|\Lda_K|\leq C(n,\sigma)\|K-\tilde c\|_{L^{2n/(n+2\sigma)}}.
\]
Thus
$$\|(\lda_KK-\Lda_K\cdot x)w_K^{\frac{n+2\sigma}{n-2\sigma}}-P_{\sigma}(1)\|_{L^2}\le C(n,\sigma,\va_2)\|K-\tilde c\|_{L^2}.$$
Since
\[
P_{\sigma}(w_K-1)=(\lda_KK-\Lda_K\cdot x)w_K^{\frac{n+2\sigma}{n-2\sigma}}-P_{\sigma}(1),
\]
by the spherical expansion of $w_K-1$ and eigenvalues of $P_{\sigma}$, it's easy to see that, for $\sigma\geq \frac{1}{2}$,
\[
\|w_K-1\|^2_{H^1}\leq \int_{\Sn} \big(P_{\sigma}(w_K-1)\big)^2 \le C(n,\sigma,\va_2)\|K-\tilde c\|^2_{L^2}.
\]
Hence \eqref{gradient estimate for wK} holds.
%\marginpar{this square both side works weird}
%(\textbf{comment:} do we need a little
%regularity of K to make $P_{\sigma}(v_K)$ to be pointwisely well-defined by using
%proposition 2.8 in \cite{silvestre} without using Besov space? since we only have $v_K\in C^{\alpha}$ for $\alpha<2\sigma$. We need a little more that $2\sigma$.)
\end{proof}

For $P\in\Sn,\ t\ge 1$, we write $v\in \M$ as $v=\varpi (w,p)=T^{-1}_{\varphi_p}w,\ w\in\M_0,\ p=sP,\ s=(t-1)/t$. Write $E_K(v)$ in the $(w,p)$ variables:
\[
I(w,p):=E_K(v)=E_{K\circ\varphi_p}(w).
\]
Consider, for each $p\in B_1$, that
\[
\min_{w\in\M_0,\ \|w-1\|_{H^{\sigma}}\leq \va_3}I(w,p)=\min_{w\in\M_0,\ \|w-1\|_{H^{\sigma}}\leq \va_3}E_{K\circ\varphi_{p}}(w).
\]
It follows from Lemma \ref{lempe5} that for $\|K-1\|_{L^{\infty}(\Sn)}\leq \va\leq \va_2$,
the minimizer exists and we denote it as $w_p$ where $p=(t-1)P/t$, $P\in \Sn$. Set $v_p=T^{-1}_{\varphi_p}w_p$.
We also know from \eqref{pe8} that $w_p>0$ on $\Sn$. As illustrated in \eqref{pe11}, we have for some $\Lda_p\in\R^{n+1}$ that
\be\label{eq:lag-mul}
P_\sigma(w_p)=(\lda_pK\circ\varphi_p-\Lda_p\cdot x)w_p^{(n+2\sigma)/(n-2\sigma)}\quad \mbox{on } \mathbb{S}^n,
\ee
where
\[
 \lda_p=\frac{\dashint_{\mathbb{S}^n}w_p P_\sigma(w_p)}{
 %\left(
 \dashint_{\mathbb{S}^n}K\circ\varphi_p w_p^{\frac{2n}{n-2\sigma}}
 %\right)^{(n-2\sigma)/n}
 }.
\]
 It follows from the Kazdan-Warner type condition \eqref{1.3} that
\[
\int_{\Sn}\langle\nabla(\lda_pK\circ\varphi_p-\Lda_p\cdot x),\nabla x\rangle w_p^{\frac{2n}{n-2\sigma}}=0.
\]
Namely,
\be\label{eq:lag-mul-2}
\sum_{j=1}^{n+1}\Lda_p^j\int_{\Sn}\langle\nabla x_j,\nabla x_i\rangle w_p^{\frac{2n}{n-2\sigma}}=\lda_p\int_{\Sn}\langle\nabla(K\circ\varphi_p),\nabla x_i\rangle w_p^{\frac{2n}{n-2\sigma}},\ 1\leq i\leq n+1.
\ee

It follows from the implicit function theorem that $w_p$ depends $C^2$ on $p$. Hence,
we have, together with the fact that $\int_{\Sn}\langle\nabla x_j,\nabla x_i\rangle w_p^{\frac{2n}{n-2\sigma}}$
is a positive definite matrix, that both $\lda_p$ and $\Lda_p$ depend $C^2$ on $p$. %\textbf{(better check by hands again.)}

Let
\[
\begin{split}
\mathcal{N}_1&=\{w\in \M_0\ |\ \|w-1\|_{H^{\sigma}}\le \va_3\},\\
\mathcal{N}_2(\tilde t)&=\{v\in \M\ |\ v=\varpi (w,p) \ \mbox{for some } w\in\mathcal{N}_1 \\
& \quad\quad\mbox{and } p=sP, P\in \Sn, s=\frac{t-1}{t}, 1\le t<\tilde t\},\\
\mathcal{N}_2&=\mathcal{N}_2(\infty),\\
\mathcal{N}_3(\tilde t)&=\{v\in H^{\sigma}\setminus\{0\}\ |\ \|v\|^{-1}_{L^{2n/(n-2\sigma)}}v\in \mathcal{N}_2(\tilde t)\},\\
\mathcal{N}_3&=\mathcal{N}_3(\infty).
\end{split}
\]

\begin{thm}\label{perturbation}
Suppose $\sigma\geq \frac{1}{2}$. There exists some constant $\va_4=\va_4(n)\in (0,\va_2)$
such that for any $T_1>0$ and any nonincreasing positive continuous function $\omega(t) (1\leq t<\infty)$
satisfying $\lim_{t\to\infty}\omega(t)=0$, if a nonconstant function $K\in C^1(\Sn)$ satisfies, for $t\geq T_1$, that
\[
\|K-1\|_{L^{\infty}(\Sn)}\leq \va_4,
\]
\be\label{condition for not degree}
\|K\circ\varphi_{P,t}-K(P)\|^2_{L^2(\Sn)}\leq \omega(t)\left|\int_{\Sn}K\circ\varphi_{P,t}(x)x\right|
\ee
for all $P\in\Sn$ and
\be\label{eq:degree not zero}
\deg\left(\int_{\Sn}K\circ\varphi_{P,t}(x)x,B^{n+1},0\right) \neq 0,% t\geq T_1.
\ee
Then \eqref{main equ} has at least one positive solution. Furthermore, for any $\al\in (0,1)$ satisfying that $\al+2\sigma$ is not an integer, there exist some positive constants $C_2$ depending only on $n,\al, \sigma, T_1$ and $\omega$ such that for all $ C\ge C_2$,
\be\label{eq: degree}
\begin{split}
&\deg(v-(P_{\sigma})^{-1}K|v|^{4\sigma/(n-2\sigma)}v,\\
&\quad\quad \mathcal{N}_3(t)\cap\{v\in C^{2\sigma+\al}\ |\ \|v\|_{C^{2\sigma+\al}}<C, 0\})\\
&\quad=(-1)^n\deg\left(\int_{\Sn}K\circ\varphi_{P,t}(x)x,B^{n+1},0\right) .
\end{split}
\ee
\end{thm}

\begin{proof}
For $P\in\Sn$ and $t\geq 1$, we set
\[
\begin{split}
&A(P,t)=\frac{1}{n}|\Sn|^{-1}\int_{\Sn}\langle\nabla(K\circ\varphi_{P,t}),\nabla x\rangle w_{p}^{2n/(n-2\sigma)},\\
&G(P,t)=|\Sn|^{-1}\int_{\Sn}K\circ\varphi_{P,t}(x)x.
\end{split}\]
It is clear that $G(P,t)\neq 0$ for all $P\in\Sn$ and $t>T_1$. We write
\[
A(P,t)=G(P,t)+I+II,
\]
where
\[
\begin{split}
I&=|\Sn|^{-1}\int_{\Sn}\big(K\circ\varphi_{P,t}-K(P)\big)x\big(w_p^{2n/(n-2\sigma)}-1\big),\\
II&=-\frac{1}{n}|\Sn|^{-1}\int_{\Sn}\big(K\circ\varphi_{P,t}-K(P)\big)\langle\nabla x, \nabla(w_p^{2n/(n-2\sigma)})\rangle.
\end{split}
\]
Using Cauchy-Schwartz inequality, \eqref{pe9}, \eqref{pe10} and \eqref{gradient estimate for wK}, we have
\[
\begin{split}
|I|&\leq C \|K\circ\varphi_{P,t}-K(P)\|_{L^2(\Sn)}\|w_p^{2n/(n-2\sigma)}-1\|_{L^2(\Sn)}\\
&\leq C \|K\circ\varphi_{P,t}-K(P)\|_{L^2(\Sn)} \|K\circ\varphi_{P,t}-K(P)\|_{L^{2n/(n+2\sigma)}(\Sn)}\\
&\leq C \omega(t)|G(p,t)|.\\
|II|&\leq C\|K\circ\varphi_{P,t}-K(P)\|_{L^2(\Sn)}\|\nabla(w_p^{2n/(n-2\sigma)})|_{L^2(\Sn)}\\
&\leq C \|K\circ\varphi_{P,t}-K(P)\|_{L^2(\Sn)}\|K\circ\varphi_{P,t}-K(P)\|_{L^2(\Sn)}\\
&\leq C \omega(t)|G(p,t)|.
\end{split}
\]
It follows immediately that for large $t$,
\[
A(P,t)\cdot G(P,t)\geq (1-C\omega(t))|G(P,t)|^2.
\]
Therefore,% for large $t$,
\[
\deg(A(P,t),B^{n+1},0)=\deg(G(P,t),B^{n+1},0).
\]
Since the matrix $[\int_{\Sn}\langle\nabla x_j,\nabla x_i\rangle w_p^{\frac{2n}{n-2\sigma}}]$ is positive definite, we have from \eqref{eq:lag-mul-2} that
\be\label{eq:degree equal}
\deg(\Lda_p, B^{n+1}_s,0)=\deg(A(P,t),B^{n+1},0)=\deg(G(P,t),B^{n+1},0)
\ee
for $s$ sufficiently closed to 1.
It follows from \eqref{eq:degree equal} and our hypothesis that for $s$ sufficiently closed to 1,
\[
\deg(\Lda_p, B^{n+1}_s,0)\neq 0.
\]
Therefore $\Lda_p$ has to have a zero inside $B^{n+1}$ which immediately implies that \eqref{main equ} has at least one positive solution.

Next we evaluate $\partial_{p}I(w_{p_0},p)|_{p=p_0}$ for $p_0=(t_0-1)P_0/t, P_0\in\Sn, t_0\ge 1$.

For $\|K-1\|_{L^{\infty}(\Sn)}\leq\va\leq\va_{4}$, for each $p\in B^{n+1}$, there exists a unique $w_p\in\M_0$, $\|w_p-1\|\leq \va_3$, such that
\[
\begin{split}
I(w_{p},p)&=\min_{w\in \M_0, \|w-1\|<\va_3}I(w,p),\\
w_p&>0\quad\mbox{on }\Sn,\\
D^2_w I(w_p,p)&\quad\mbox{is positive definite},\\
\|w_p-1\|_{H^{\sigma}}&\leq C\|K\circ\varphi_{P,t}-K(p)\|_{L^{2n/(n+2\sigma)}(\Sn)},\\
|w_K-1|+|P_\sigma(w_K-1)|&\leq o_{\va}(1).
\end{split}
\]
It can be seen from \eqref{eq:lag-mul} that $v_{p_0}=T^{-1}_{\varphi_{p_0}}w_{p_0}$ satisfies
\[
P_\sigma(v_{p_0})=(\lda_{p_0}K-\Lda_{p_0}\cdot \varphi^{-1}_{p_0})v_{p_0}^{(n+2\sigma)/(n-2\sigma)}\quad \mbox{on } \mathbb{S}^n.
\]
It follows that for any $\psi\in C^{\infty}(\Sn)$, we have
\[
\begin{split}
&\partial_v E_K(v_{p_0})\psi=-2\left(\dashint_{\Sn}Kv_{p_0}^{2n/(n-2\sigma)}\right)^{\frac{2\sigma-n}{n}}\dashint_{\Sn}\Lda_{p_0}\cdot \varphi^{-1}_{p_0}v_{p_0}^{\frac{n+2\sigma}{n-2\sigma}}\psi\\
&\partial_{p}I(w_{p_0},p)|_{p=p_0}\\
&=\partial_v E_K(v_{p_0})(\pa_p(T^{-1}_{\varphi_p}w_{p_0})|_{p=p_0})\\
&=-2\left(\dashint_{\Sn}Kv_{p_0}^{2n/(n-2\sigma)}\right)^{\frac{2\sigma-n}{n}}\dashint_{\Sn}\Lda_{p_0}\cdot \varphi^{-1}_{p_0}v_{p_0}^{\frac{n+2\sigma}{n-2\sigma}}(\pa_p(T^{-1}_{\varphi_p}w_{p_0})|_{p=p_0})\\
&=-\frac{n-2\sigma}{n}\left(\dashint_{\Sn}Kv_{p_0}^{2n/(n-2\sigma)}\right)^{\frac{2\sigma-n}{n}}\dashint_{\Sn}\Lda_{p_0}\cdot \varphi^{-1}_{p_0}(\pa_p(T^{-1}_{\varphi_p}w_{p_0})^{2n/(n-2\sigma)}|_{p=p_0})\\
&=-\frac{n-2\sigma}{n}\left(\dashint_{\Sn}Kv_{p_0}^{2n/(n-2\sigma)}\right)^{\frac{2\sigma-n}{n}}\pa_p\left(\dashint_{\Sn}\Lda_{p_0}\cdot \varphi^{-1}_{p_0}(T^{-1}_{\varphi_p}w_{p_0})^{\frac{2n}{n-2\sigma}}\right)\bigg|_{p=p_0}\\
&=-\frac{n-2\sigma}{n}\left(\dashint_{\Sn}Kv_{p_0}^{2n/(n-2\sigma)}\right)^{\frac{2\sigma-n}{n}}\pa_p\left(\dashint_{\Sn}\Lda_{p_0}\cdot \varphi^{-1}_{p_0}\circ\varphi_pw_{p_0}^{\frac{2n}{n-2\sigma}}\right)\bigg|_{p=p_0}.
\end{split}
\]
By Appendix A in \cite{Li95}, the matrix
\[
\pa_p\left(\dashint_{\Sn}\Lda_{p_0}\cdot \varphi^{-1}_{p_0}\circ\varphi_pw_{p_0}^{\frac{2n}{n-2\sigma}}\right)\bigg|_{p=p_0}
\]
is invertible with positive determinant. Therefore, for $t$ large with $s=(t-1)/t$, we have
\be\label{eq:degree-1}
(-1)^{n+1}\deg(\Lda_{p}, B_s, 0)=\deg(\pa_p I(w_p,p), B_s, 0).
\ee
Given Theorem B.1 in \cite{Li95} and Appendix \ref{spaces}, the rest of the proof of \eqref{eq: degree} is similar to that in page 386 of \cite{Li95} and we omit them here.
%Let $\tau>0$ be sufficiently small and we introduce, for $v=\varpi(w,p)$, that
%\[
%\begin{split}
%I_{\tau}(w,p)&=E_{K,\tau}(v),\\
%E_{K,\tau}(v)&=\frac{\dashint_{\mathbb{S}^n}vP_\sigma(v)\,\ud v_{g_{\Sn}}}{\left(\dashint_{\mathbb{S}^n}K|v|^{\frac{2n}{n-2\sigma}-\tau}\,\ud v_{g_{\Sn}}\right)^{\frac{2(n-2\sigma)}{2n-\tau(n-2\sigma)}}}.
%\end{split}
%\]
\end{proof}

Next we will give sufficient conditions for $K$ to satisfy \eqref{condition for not degree}.
The proof of Lemma 6.6 in \cite{Li95} indeed shows the following

\begin{lem}\label{lem condition for not degree}
Suppose $K\in C^{1,1}(\Sn)$ satisfies for some constant $A_1>0$, $K(P)\geq A_1$ for all $P\in\Sn$, and there exists some constant $0<\va_1<1$, such that for each critical point $P_0\in\Sn$ of $K$, there exists some $\beta=\beta(P_0)\in(1,n)$ such that in some geodesic normal coordinate system centered at $P_0$,
\[
K(y)=K(0)+Q^{(\beta)}(y)+R(y),\quad |y|<\va_1,
\]
where $Q^{(\beta)}(\lda y)=\lda^{\beta}Q^{(\beta)}(y)$ for any $\lda>0$, $y\in\R^n$ and
\[
A_6|y|^{\beta-1}\leq |\nabla Q^{(\beta)}(y)|\leq A_7|y|^{\beta-1},
\]
for some positive constants $A_6, A_7$. Here $R(y)$ satisfies $|R(y)||y|^{-\beta}+|\nabla R(y)||y|^{1-\beta}\leq \eta(|y|)$ for some continuous function $\eta$ with $\lim_{r\to 0^+}\eta(r)=0$. Suppose also that for some constant $d>0$, $|\nabla K(P)|\geq d$ for all $P\in\Sn$ with $\min\{|P-P_0||\nabla K(P_0)=0\}\ge \va_1/20$. Then there exists some positive constant $C_3$ depending on $n$, $A_1$, $A_6$, $A_7$, $d$, $\va_1$, $\min\{\beta-1, n-\beta\}$, $\eta$ and the modulo of continuity of $\nabla K$, such that for $P\in\Sn$, $\min\{|P-P_0||\nabla K(P_0)=0\}\ge C_3/t$, we have
\be\label{eq:lem condition for not degree}
\|K\circ\varphi_{P,t}-K(P)\|^2_{L^2(\Sn)}\leq o\left(\left|\int_{\Sn}K\circ\varphi_{P,t}(x)x\right|\right)
\ee
as $t\to\infty$.
\end{lem}

The following is Lemma 6.7 in \cite{Li95}.
\begin{lem}\label{lem condition degree not zero}
Suppose $K\in C^{1,1}(\Sn)$ satisfies for some constant $A_1>0$, $K(P)\geq A_1$ for all $P\in\Sn$, and there exists some constant $0<\va_1<1$, such that for each critical point $P_0\in\Sn$ of $K$, there exists some $\beta=\beta(P_0)\in(1,n)$ such that in some geodesic normal coordinate system centered at $P_0$,
\[
K(y)=K(0)+Q^{(\beta)}(y)+R(y),\quad |y|<\va_1,
\]
where $Q^{(\beta)}(\lda y)=\lda^{\beta}Q^{(\beta)}(y)$ for any $\lda>0$, $y\in\R^n$ and
\[
A_6|y|^{\beta-1}\leq |\nabla Q^{(\beta)}(y)|\leq A_7|y|^{\beta-1},
\]
for some positive constants $A_6, A_7$, $R(y)$ denotes some quantity satisfying\\ $\lim_{y\to 0}R(y)|y|^{-\beta} =0$ and $\lim_{y\to 0}|\nabla R(y)||y|^{1-\beta}=0$. Suppose also that for some constant $d>0$, $|\nabla K(P)|\geq d$ for all $P\in\Sn$ with $\min\{|P-P_0||\nabla K(P_0)=0\}\ge \va_1/20$, and
\[
\left(
\begin{array}{l}
 \int_{\R^n}\nabla Q^{(\beta)}(y+\eta)(1+|y|^2)^{-n}\,\ud y\\[2mm]
\int_{\R^n}Q^{(\beta)}(y+\eta)\frac{|y|^2-1}{|y|^2+1}(1+|y|^2)^{-n}\,\ud y
\end{array} \right)\neq 0 \quad \forall\  \eta\in \R^n,
\]
Then \eqref{eq:lem condition for not degree} holds. In particular, if $K$ is not identically equal to a constant, we have
\[
\int_{\Sn}K\circ\varphi_{P,t}(x)x\neq 0
\]
for large $t$.
\end{lem}

\begin{cor}\label{cor1 of existence}
Suppose $K\in C^{1,1}(\Sn)$ is some positive function satisfying for each critical point $P_0\in\Sn$ of $K$, there exists some $\beta=\beta(P_0)\in(1,n)$ such that in some geodesic normal coordinate system centered at $P_0$,
\[
K(y)=K(0)+Q^{(\beta)}(y)+R(y),\quad \mbox{for all }y\mbox{ close to }0,
\]
where $Q^{(\beta)}(\lda y)=\lda^{\beta}Q^{(\beta)}(y)$ for any $\lda>0$, $y\in\R^n$ and
$R(y)$ denotes some quantity satisfying $\lim_{y\to 0}R(y)|y|^{-\beta} =0$ and $\lim_{y\to 0}|\nabla R(y)||y|^{1-\beta}=0$. Suppose also that
\[
|\nabla Q^{(\beta)}(y)|\sim |y|^{\beta-1}\quad \mbox{for all }y\mbox{ close to }0,
\]
\[
\left(
\begin{array}{l}
 \int_{\R^n}\nabla Q^{(\beta)}(y+\eta)(1+|y|^2)^{-n}\,\ud y\\[2mm]
\int_{\R^n}Q^{(\beta)}(y+\eta)\frac{|y|^2-1}{|y|^2+1}(1+|y|^2)^{-n}\,\ud y
\end{array} \right)\neq 0, \quad \forall\  \eta\in \R^n,
\]
Then for $t$ large enough,
\[
\int_{\Sn}K\circ\varphi_{P,t}(x)x\neq 0\quad\mbox{for all}\quad P\in\Sn.
\]
If we further assume that
\[
\|K-1\|_{L^{\infty}(\Sn)}\leq \va_4
\]
and
\[
\deg\left(\int_{\Sn}K\circ\varphi_{P,t}(x)x,B^{n+1},0\right) \neq 0,
\]
%for large $t$,
then for $\sigma\in [1/2,1)$, \eqref{main equ} has at least one positive $C^2$ solution $v$. Furthermore, for every $0<\al<1$ satisfying that $\al+2\sigma$ is not an integer and sufficiently large positive constant $C_2$, we have
\[
\begin{split}
&\deg(v-(P_{\sigma})^{-1}K|v|^{4\sigma/(n-2\sigma)}v,\\
&\quad\quad \mathcal{N}_3(t)\cap\{v\in C^{2\sigma+\al}\ |\ \|v\|_{C^{2\sigma+\al}}<C_2, 0\})\\
&\quad=(-1)^n\deg\left(\int_{\Sn}K\circ\varphi_{P,t}(x)x,B^{n+1},0\right) .
\end{split}
\]
\end{cor}
\begin{proof}
It follows from Theorem \ref{perturbation} and Lemma \ref{lem condition degree not zero}.
\end{proof}

\begin{cor}\label{cor2 of existence}
Suppose $K\in C^{1,1}(\Sn)$ is some positive function satisfying for each critical point $P_0\in\Sn$ of $K$, there exists some $\beta=\beta(P_0)\in(1,n)$ such that in some geodesic normal coordinate system centered at $P_0$,
\[
K(y)=K(0)+Q^{(\beta)}(y)+R(y),\quad \mbox{for all }y\mbox{ close to }0,
\]
where $Q^{(\beta)}(\lda y)=\sum_{j=1}^{n}a_{j}|y_j|^{\beta}$, $a_j=a_j(\xi_0)\neq 0$, $\sum_{j=1}^n a_j\neq 0$, and
$R(y)$ denotes some quantity satisfying $\lim_{y\to 0}R(y)|y|^{-\beta} =0$ and $\lim_{y\to 0}|\nabla R(y)||y|^{1-\beta}=0$. Then for $t$ large enough,
\[
\int_{\Sn}K\circ\varphi_{P,t}(x)x\neq 0\quad\mbox{for all}\quad P\in\Sn,
\]
and
\[
\begin{split}
&\deg\left(\int_{\Sn}K\circ\varphi_{P,t}(x)x,B^{n+1},0\right)\\
&\quad=\sum_{\xi\in \Sn\mbox{ such that }\nabla_{g_{\Sn}}K(\xi)=0,\  \sum_{j=1}^na_j(\xi)<0}(-1)^{i(\xi)}- (-1)^n,
\end{split}
\]
where
\[
 i(\xi)=\#\{a_j(\xi): \nabla_{g_{\Sn}}K(\xi)=0,a_j(\xi)<0,1\leq j\leq n\}.
\]
If we further assume that
\[
\|K-1\|_{L^{\infty}(\Sn)}\leq \va_4
\]
and
\[
\sum_{\xi\in \Sn\mbox{ such that }\nabla_{g_{\Sn}}K(\xi)=0,\  \sum_{j=1}^na_j(\xi)<0}(-1)^{i(\xi)}\neq(-1)^n,
\]
then for $\sigma\in [1/2,1)$, \eqref{main equ} has at least one positive $C^2$ solution $v$. Furthermore, for every $0<\al<1$ satisfying that $\al+2\sigma$ is not an integer and sufficiently large positive constant $C_2$, we have
\[
\begin{split}
&\deg(v-(P_{\sigma})^{-1}K|v|^{4\sigma/(n-2\sigma)}v,\\
&\quad\quad \mathcal{N}_3(t)\cap\{v\in C^{2\sigma+\al}\ |\ \|v\|_{C^{2\sigma+\al}}<C_2, 0\})\\
&\quad=(-1)^n\sum_{\xi\in \Sn\mbox{ such that }\nabla_{g_{\Sn}}K(\xi)=0,\  \sum_{j=1}^na_j(\xi)<0}(-1)^{i(\xi)}-1.
\end{split}
\]
\end{cor}
\begin{proof}
It follows from the proof of Corollary 6.2 in \cite{Li95} and Corollary \ref{cor1 of existence}.
\end{proof}

\begin{thm}\label{thm:biggest one}
Let $\sigma \in (0,1)$. Suppose that $K\in C^{1,1}(\Sn)$, for some constant $A_1>0$,
\[
 1/A_1\leq K(\xi)\leq A_1\quad  \mbox{for all }\xi\in \Sn.
\]
Suppose also that for any critical point $\xi_0$ of $K$,
under the stereographic projection coordinate system $\{y_1, \cdots, y_n\}$ with $\xi_0$ as south pole, there exist
some small neighborhood $\mathscr{O}$ of $0$, a positive constant $L$, and $\beta=\beta(\xi_0)\in(n-2\sigma, n)$ such that
\[
 \|\nabla^{[\beta]}K\|_{C^{\beta-[\beta]}(\mathscr{O})}\leq L
\]
and
\[
 K(y)=K(0)+Q_{(\xi_0)}^{(\beta)}(y)+R_{(\xi_0)}(y)\quad \mbox{in }\mathscr{O},
\]
where $Q_{\xi_0}^{(\beta)}(y)\in C^{[\beta]-1,1}(\mathbb{S}^{n-1})$ satisfies
$Q_{\xi_0}^{(\beta)}(\lda y)=\lda^{\beta}Q_{\xi_0}^{(\beta)}(y)$, $\forall \lda >0$, $y\in \R^n$,
\[
 |\nabla Q^{(\beta)}(y)|\sim |y|^{\beta-1}\quad y\in \mathscr{O},
\]
\[
\left(
\begin{array}{l}
 \int_{\R^n}\nabla Q^{(\beta)}(y+\eta)(1+|y|^2)^{-n}\,\ud y\\[2mm]
\int_{\R^n}Q^{(\beta)}(y+\eta)(1+|y|^2)^{-n}\,\ud y
\end{array} \right)\neq 0, \quad \forall\  \eta\in \R^n,
\]
and
\[
\left(
\begin{array}{l}
 \int_{\R^n}\nabla Q^{(\beta)}(y+\eta)(1+|y|^2)^{-n}\,\ud y\\[2mm]
\int_{\R^n}Q^{(\beta)}(y+\eta)\frac{|y|^2-1}{|y|^2+1}(1+|y|^2)^{-n}\,\ud y
\end{array} \right)\neq 0, \quad \forall\  \eta\in \R^n,
\]
and $R_{\xi_0}(y)\in C^{[\beta]-1,1}(\mathscr{O})$ satisfies $\lim_{y\to 0}\sum_{s=0}^{[\beta]}|\nabla^sR|{\xi_0}(y)|y|^{-\beta+s}=0$. Then for any $\va, \delta>0$, there exists a positive constant $C(K, n,\delta,\va)$ such that for all $\va\le\mu\le 1$ and for all $\delta\le\sigma\le 1-\delta$, every positive solution $v$ of \eqref{main equ} with $K$ replaced by $K_\mu=\mu K+(1-\mu)$ satisfies
\be\label{eq:final1}
 1/C(K,n,\delta,\va)\leq v\leq C(K,n,\delta,\va)\quad \mbox{on }\Sn.
\ee
Also, for large $t$,
\be\label{eq:final2}
\int_{\Sn}K\circ\varphi_{P,t}(x)x\neq 0 \quad\mbox{for all}\quad P\in \Sn.
\ee
If we further assume that
\be\label{eq:final3}
\deg\left(\int_{\Sn}K\circ\varphi_{P,t}(x)x,B^{n+1},0\right)\neq 0,
\ee
then \eqref{main equ} has at least one $C^2$ positive solution.
\end{thm}
\begin{proof}
Given that $\delta\le \sigma\le 1-\delta$, the estimates established in \cite{JLX} depend on $\delta$ instead of $\sigma$. Hence, \eqref{eq:final1} has actually been proved in \cite{JLX}. \eqref{eq:final2} follows from Lemma \ref{lem condition degree not zero}. In the following we will show the existence part. We first consider the case $\sigma\in[1/2,1)$.

Claim: there exists some constant $\va_7>0$ such that for $0\le\mu\le\va_7$ we have $\|K_u-1\|_{L^{\infty}(\Sn)}<\va_4$, and if $v$ is any solution of \eqref{main equ} with $K=K_\mu (0\le \mu\le \va_7)$ and $\varpi^{-1}v= (w,p)$, $(w,p)\in\M_0\times B^{n+1}$, then $w\in\mathcal{N}_1$.

This claim can be proved by contradiction. Suppose along a subsequence of $\mu\to 0$, there exists $v_\mu$ satisfying \eqref{main equ} with $K=K_\mu$, but $\|w-1\|_{H^{\sigma}}\ge \va_3$ where $\varpi^{-1}v_\mu= (w_\mu,p_\mu)$, $(w_\mu,p_\mu)\in\M_0\times B^{n+1}$. It follows from Theorem 5.2 in \cite{JLX} that after passing to a subsequence, either $\{v_\mu\}$ stays bounded in $L^{\infty}(\Sn)$ or it has precisely one isolated simple blow up point. It is clear that $w_\mu$ satisfies
\be\label{eq:wmu}
P_{\sigma}(w_\mu)=c(n,\sigma)(K_\mu\circ\varphi_\mu)w_\mu^{(n+2\sigma)/(n-2\sigma)},
\ee
where $\varphi_\mu$ is the conformal transformation corresponding to $p_\mu$. It follows that $w_\mu\in C^{2\sigma+\al}$ for any $\al\in (0,1)$ satisfying that $2\sigma+\al$ is not an integer.

It is not difficult to see from the estimates on isolated simple blow up point in \cite{JLX}, \eqref{eq:wmu} and local estimates established in \cite{JLX} that in either case we have, after passing to a subsequence,
\[
w_\mu\to w \quad \mbox{in } C^{\beta}(\Sn)\quad\mbox{and}\quad w_\mu\rightharpoonup w \quad\mbox{weakly in } H^{\sigma}(\Sn)
\]
for some $w\in \M_0, w>0, \beta\in (0,1)$. Sending $\mu$ to $0$, we have
\[
P_{\sigma}(w)=c(n,\sigma)w^{(n+2\sigma)/(n-2\sigma)}.
\]
It follows that $w\equiv 1$. Using \eqref{eq:wmu} again, we have
\[
\|w_\mu-1\|_{H^{\sigma}}=\int_{\Sn} (w_\mu-1)P_{\sigma}(w_\mu-1)\leq C\int_{\Sn}|w_\mu-1|\to 0\quad \mbox{as}\quad \mu\to 0.
\]
This is a contradiction. The claim is proved.

On the other hand, it follows from Theorem 5.2 and Theorem 5.3 in \cite{JLX} and the Harnack inequality that there exists some constant $C^*>1$ such that for all $\va_7\le\mu\le 1$,
\[
1/C^* \le v_\mu \le C^*.
\]
where $v_\mu$ is any solution of \eqref{main equ} with $K=K_\mu$.

It follows from the homotopy property of the Leray Schauder degree and Corollary \ref{cor1 of existence} that
\[
\begin{split}
&\deg(v-(P_{\sigma})^{-1}Kv^{(n+2\sigma)/(n-2\sigma)}, C^{2\sigma+\al}(\Sn)\cap\{1/C^* \le v_\mu \le C^*\},0)\\
&\quad=\deg(v-(P_{\sigma})^{-1}K_{\va_7}v^{(n+2\sigma)/(n-2\sigma)}, C^{2\sigma+\al}(\Sn)\cap\{1/C^* \le v_\mu \le C^*\},0)\\
&\quad=(-1)^n\deg\left(\int_{\Sn}K_{\va_7}\circ\varphi_{P,t}(x)x,B^{n+1},0\right)\\
&\quad=(-1)^n\deg\left(\int_{\Sn}K\circ\varphi_{P,t}(x)x,B^{n+1},0\right)\\
&\quad\neq 0.
\end{split}
\]
The existence of solutions of \eqref{main equ} for $\sigma\ge 1/2$ follows immediately.

For the case $\sigma\in (0,1/2)$, we consider the problem for $\sigma_t=t\sigma+2(1-t)/3$, and the existence for $\sigma$ follows from a degree argument.
%\marginpar{not finished}
\end{proof}

\begin{proof}[Proof of the existence part of Theorem \ref{general exist}]
It follows from Theorem \ref{thm:biggest one} and Corollary \ref{cor2 of existence}.
\end{proof}

\section{A fractional Aubin inequality}\label{improved inequality}
Let
\[
 \M^p=\left\{v \in H^{\sigma}(\mathbb{S}^n):\dashint_{\mathbb{S}^n}|v|^{p}\,\ud v_{g_{\Sn}}=1\right\},
\]
\[
 \M_0^p=\left\{v\in \M:\dashint_{\mathbb{S}^n} x|v|^{p}\,\ud v_{g_{\Sn}}=0\right\}.
\]
The Sobolev inequality \eqref{pe1} states that
\[
\min_{v\in \M^{\frac{2n}{n-2\sigma}} }\dashint vP_{\sigma}(v)\geq P_{\sigma}(1).
\]
\begin{prop}\label{aubin inequality}
For $\sigma\in (0,1)$, $n\geq 2$, $2<p\leq\frac{2n}{n-2\sigma}$, given any $\va>0$, there exists some constant $C_{\va}\ge 0$ such that
\be\label{eq:aubin inequality}
\inf_{v\in \M^{p}_0}\left\{2^{\frac{2}{p}-1}(1+\va)\dashint_{\Sn} v P_{\sigma}(v)+C_{\va}\dashint _{\Sn}v^2\right\}\geq P_{\sigma}(1).
\ee
\end{prop}

When $\sigma=1$, the above proposition was proved by Aubin \cite{Au79}. See also \cite{D1} for such inequality in some higher order Sobolev spaces. %We will adapt the proof of Theorem 2 in \cite{Au79} to show \eqref{eq:aubin inequality}.

\begin{proof}
 First of all, by H\"older inequality, \eqref{pe1} and \eqref{description of P sigma}, we have for all $v\in H^{\sigma}(\Sn)$,
\be\label{p-sobolev}
\begin{split}
\left(\int_{\Sn}v^p\right)^{\frac 2p}&\leq K^2\int_{\Sn}vP_{\sigma}(v)\\
&=K^2P_{\sigma}(1)\int_{\Sn}v^2+\frac{K^2c_{n,-\sigma}}{2}\iint_{\Sn\times\Sn}\frac{(v(x)-v(y))^2}{|x-y|^{n+2\sigma}},
\end{split}
\ee
where $K^2:=|\Sn|^{\frac 2p -1}(P_{\sigma}(1))^{-1}$.
 Let $\eta\in(0,\frac12)$ be chosen later. Let $\Lambda$ be the space of first spherical harmonics. As shown in \cite{Au79}, there exists $\{\xi_i\}_{i=1,\cdots,k}\subset\Lambda$ such that $1+\eta<\sum_{i=1}^k|\xi_i|^{\frac 2p}<1+2\eta$ with $|\xi_i|<2^{-p}$. Let $h_i\in C^1(\Sn)$ be such that $h_i\xi_i\ge 0$ on $\Sn$ and
\[
\big||h_i|^2-|\xi_i|^{\frac 2p}\big|<\left(\frac\eta k\right)^p.
\]
Then
\[
1<\sum_{i=1}^k|h_i|^2<1+3\eta,
\]
and by the mean value theorem
\[
\big||h_i|^p-|\xi_i|\big|\leq\frac p2 \left(\frac\eta k\right)^p.
\]
For any nonnegative $v\in H^{\sigma}(\Sn)$, we have,
\[
\begin{split}
\left(\int_{\Sn}v^p \right)^{\frac 2p}=\|v^2\|_{L^{\frac p2}(\Sn)}&\le \|\sum_{i=1}^k|h_i|^2v^2\|_{L^{\frac p2}(\Sn)}\\
&\le \sum_{i=1}^k\||h_i|^2v^2\|_{L^{\frac p2}(\Sn)}= \sum_{i=1}^k\left(\int_{\Sn}|h_i|^pv^p \right)^{\frac 2p}.
\end{split}
\]
Given $f: \Sn\to\R$, denote $f_+=\max(f,0)$ and $f_-=\max(-f,0)$. For $v\in \M_0^p$, one has that
\[
\int_{\Sn}\xi_{i+}v^p=\int_{\Sn}\xi_{i-}v^p.
\]
Hence for a nonnegative function $v\in \M_0^p$, it follows from \eqref{p-sobolev} and $h_i\xi_i\ge 0$ that
\[
\begin{split}
\left(\int_{\Sn}|h_i|^pv^p \right)^{\frac 2p}&=\left(\int_{\Sn}h_{i+}^pv^p +\int_{\Sn}h_{i-}^pv^p \right)^{\frac 2p}\\
&\leq \left(\int_{\Sn}\xi_{i+}v^p +\va_0^p v^p+\int_{\Sn}h_{i-}^pv^p \right)^{\frac 2p}\\
&\leq 2^{\frac 2p}\left(\int_{\Sn}\va_0^p v^p+\int_{\Sn}h_{i-}^pv^p \right)^{\frac 2p}\\
&\leq 2^{\frac 2p}\left(\int_{\Sn}(\va_0+h_{i-})^pv^p \right)^{\frac 2p}\\
&\leq 2^{\frac 2p}\left(K^2P_{\sigma}(1)\int_{\Sn}(h_{i-}+\va_0)^2v^2)+\frac{K^2c_{n,-\sigma}}{2} I\right).
\end{split}
\]
where $\va_0=(\frac p2)^{1/p}\frac {\eta}{k}$,
\[
\begin{split}
I&=\iint_{\Sn\times\Sn}\frac{((h_{i-}(x)+\va_0)v(x)-(h_{i-}(y)+\va_0)v(y))^2}{|x-y|^{n+2\sigma}}.
\end{split}
\]
Since
\[
\begin{split}
&((h_{i-}(x)+\va_0)v(x)-(h_{i-}(y)+\va_0)v(y))^2\\
&=(h_{i-}(x)-h_{i-}(y))^2v(x)^2 +(h_{i-}(y)+\va_0)^2(v(x)-v(y))^2\\
&\quad +2(h_{i-}(x)-h_{i-}(y))v(x)(h_{i-}(y)+\va_0)(v(x)-v(y)),\\
\end{split}
\]
we have
\[
\begin{split}
I%&=\iint_{\Sn\times\Sn}\frac{((h_{i-}(x)+\va_0)v(x)-(h_{i-}(y)+\va_0)v(y))^2}{|x-y|^{n+2\sigma}}\\
%&=\int_{\Sn}v^2(x)\int_{\Sn}\frac{(h_{i-}(x)-h_{i-}(y))^2}{|x-y|^{n+2\sigma}}+\int_{\Sn}(h_{i-}(y)+\va_0)^2\int_{\Sn}\frac{(v(x)-v(y))^2}{|x-y|^{n+2\sigma}}\\
%&\quad+2\iint_{\Sn\times\Sn}\frac{(h_{i-}(x)-h_{i-}(y))v(x)h_{i-}(y)(v(x)-v(y))}{|x-y|^{n+2\sigma}}\\
&\leq \int_{\Sn}v^2(x)\int_{\Sn}\frac{(h_{i-}(x)-h_{i-}(y))^2}{|x-y|^{n+2\sigma}}+\int_{\Sn}(h_{i-}(y)+\va_0)^2\int_{\Sn}\frac{(v(x)-v(y))^2}{|x-y|^{n+2\sigma}}\\
&\quad+2C_1\left(\iint_{\Sn\times\Sn}\frac{(v(x)-v(y))^2}{|x-y|^{n+2\sigma}}\right)^{\frac 12}\left(\iint_{\Sn\times\Sn}\frac{v^2(x)(h_{i-}(x)-h_{i-}(y))^2}{|x-y|^{n+2\sigma}}\right)^{\frac 12}\\
%&\leq C\int_{\Sn}v^2+\int_{\Sn}(h_{i-}+\va_0)^2(y)\int_{\Sn}\frac{(v(x)-v(y))^2}{|x-y|^{n+2\sigma}}\\
%&\quad +C\left(\iint_{\Sn\times\Sn}\frac{(v(x)-v(y))^2}{|x-y|^{n+2\sigma}}\right)^{\frac 12}\left(\int_{\Sn}v^2\right)^{\frac 12}\\
&\leq C_2\int_{\Sn}v^2+\int_{\Sn}(h_{i-}(y)+\va_0)^2\int_{\Sn}\frac{(v(x)-v(y))^2}{|x-y|^{n+2\sigma}}\\
&\quad+\frac{\eta}{k}\iint_{\Sn\times\Sn}\frac{(v(x)-v(y))^2}{|x-y|^{n+2\sigma}}+C_\eta \int_{\Sn}v^2(x)\int_{\Sn}\frac{(h_{i-}(x)-h_{i-}(y))^2}{|x-y|^{n+2\sigma}}\\
&\leq C\int_{\Sn}v^2+\int_{\Sn}(h_{i-}(y)+\va_0)^2\int_{\Sn}\frac{(v(x)-v(y))^2}{|x-y|^{n+2\sigma}}+\frac{\eta}{k}\iint_{\Sn\times\Sn}\frac{(v(x)-v(y))^2}{|x-y|^{n+2\sigma}},
\end{split}
\]
where in the second inequality we have used Cauchy-Schwarz inequality, $C_1=\max|h+\va_0|^2$, $C_2=\max\int_{\Sn}\frac{(h_{i-}(x)-h_{i-}(y))^2}{|x-y|^{n+2\sigma}}\ud y$, $C_\eta>0$ depends only on $C_1$ and $\eta$, $C=C_2+C_2C_\eta$.
Also we can do exactly the same in terms of $h_{i+}$. Hence
\[
\begin{split}
&2\left(\int_{\Sn}v^p\right)^{\frac 2p}\\
&\leq 2^{\frac{2}{p}}\sum_{i=1}^k\frac{K^2c_{n,-\sigma}}{2}\int_{\Sn}\big((h_{i-}(y)+\va_0)^2+(h_{i+}(y)+\va_0)^2\big)\int_{\Sn}\frac{(v(x)-v(y))^2}{|x-y|^{n+2\sigma}}\\
&+2^{\frac{2}{p}}\sum_{i=1}^k(2\frac{\eta}{k})\iint_{\Sn\times\Sn}\frac{(v(x)-v(y))^2}{|x-y|^{n+2\sigma}}+C\int_{\Sn}v^2.
\end{split}
\]
Hence for any $\va>0$, we can choose $\eta$ sufficiently small such that
\[
\begin{split}
\left(\int_{\Sn}v^p\right)^{\frac 2p}&\leq 2^{\frac 2p -2}(K^2c_{n,-\sigma}+\va)\iint_{\Sn\times\Sn}\frac{(v(x)-v(y))^2}{|x-y|^{n+2\sigma}}+C\int_{\Sn}v^2\\
&=2^{\frac 2p -1}(K^{2}+\va c^{-1}_{n,-\sigma})(\int_{\Sn}vP_{\sigma}(v)-P_{\sigma}(1)\int_{\Sn}v^2)+C\int_{\Sn}v^2.
\end{split}
\]
Then the proposition follows immediately from the above and that for $v\in H^{\sigma}(\Sn)$,
\[
\int_{\Sn}|v|P(|v|)\leq \int_{\Sn}vP(v).
\]
\end{proof}

\begin{prop}\label{aubin-sobolev inequality}
For $n\ge 2$, there exist some constants $a^*<1$ and some $p^*<\frac{2n}{n-2\sigma}$ both of which depend only on $n$ and $\sigma$, such that for all $p^*\leq p\leq \frac{2n}{n-2\sigma}$,
\be\label{eq:aubin-sobolev}
\inf_{v\in \M^p_0} a^*\dashint_{\Sn} v P_{\sigma}(v)+(1-a^*)P_{\sigma}(1)\dashint_{\Sn} v^2\geq P_{\sigma}(1).
\ee
\end{prop}

When $\sigma=1$, the above proposition was proved by Chang and Yang \cite{CY} (see \cite{Li96} for another proof). See also \cite{D1} for such inequality in some higher order Sobolev spaces. Here we adapt the arguments in \cite{Li96} to show \eqref{eq:aubin-sobolev}.
\begin{proof}
For $v\in H^{\sigma}(\Sn), a>0$, set
\[
I_a(v)=a\dashint_{\Sn} v P_{\sigma}(v)+(1-a)P_{\sigma}(1)\dashint_{\Sn} v^2
\]
and
\[
m_{a,p}=\inf_{v\in \M^p_0} I_a(v).
\]
By standard variational methods, $m_{a,p}$ is achieved for $a>0$ and $2\leq p<\frac{2n}{n-2\sigma}$. Moreover, it is easy to see that
\be\label{eq:lim-sub}
\begin{split}
m_{a,p}&\leq P_{\sigma}(1)\quad\mbox{for all }0\leq a\leq 1,\ 2\leq p\leq\frac{2n}{n-2\sigma},\\
\lim\limits_{a\to 1}m_{a,p}&=P_{\sigma}(1)\quad\mbox{uniformly for }2\leq p\leq \frac{2n}{n-2\sigma}.
\end{split}
\ee
Indeed, the inequality \eqref{eq:lim-sub} follows from by taking the test function $v\equiv 1$. The equality in \eqref{eq:lim-sub} follows from Sobolev inequality and H\"older inequality.

We argue by contradiction. Suppose that \eqref{eq:aubin-sobolev} fails. Then there exist sequences $\{a_k\}$, $\{p_k\}\subset \R$, $\{v_k\}\subset \M^{p_k}_0$, such that $a_k<1, a_k\to 1, p_k<\frac{2n}{n-2\sigma}, p_k\to\frac{2n}{n-2\sigma}, v_k\geq 0$ and
\be\label{eq:contradiction}
I_{a_k}(v_k)=m_{a_k,p_k}<P_{\sigma}(1).
\ee
By \eqref{eq:contradiction} and \eqref{eq:aubin inequality}, there exists some positive constant $C(n,\sigma)$ independent of $k$ such that
\[
\|v_k\|_{H^{\sigma}(\Sn)}\leq C(n,\sigma),\quad\int_{\Sn}v_k^2\geq 1/C(n,\sigma).
\]
After passing to a subsequence, we have that $v_k\to \bar v$ weakly in $H^{\sigma}(\Sn)$ for some $\bar v\in H^{\sigma}(\Sn)\setminus\{0\}$.

The Euler-Lagrange equation (see, e.g., \eqref{pe11}) satisfied by $v_k$ is
\be\label{eq:minimizer-contrdiction}
a_kP_{\sigma}(v_k)+(1-a_k)P_{\sigma}(1)v_k=m_kv_k^{p_k-1}+\Lambda_k\cdot xv_k^{p_k-1},
\ee
where $m_k=m_{a_k,p_k}$ and $\Lambda_k\in \R^{n+1}$. Multiplying \eqref{eq:minimizer-contrdiction} by $v_k$ and integrating over $\Sn$, we have, by \eqref{eq:lim-sub}
\be\label{eq:lim-k}
\lim\limits_{k\to\infty}\left(\dashint_{\Sn}v_kP_{\sigma}(v_k)\right)=P_{\sigma}(1).
\ee

We claim that $|\Lambda_k|=O(1)$. Suppose the contrary, we let $\xi_k=\Lambda_k/|\Lambda_k|$ and after passing to a subsequence $\xi=\lim_{k\to\infty}\xi_k\in\Sn$. Let $\eta\in C^{\infty}(\Sn)$ be any smooth test function. Multiplying \eqref{eq:minimizer-contrdiction} by $\eta/|\Lambda_k|$, integrating it over $\Sn$ and sending $k\to\infty$, we have $\int_{\Sn}\xi\cdot x \bar v^{\frac{n+2\sigma}{n-2\sigma}}\eta=0$. Hence $\bar v=0$ which is a contradiction.

It is clear that $\bar v$ satisfies
\[
P_{\sigma}(\bar v)=P_{\sigma}(1)\bar v^{\frac{n+2\sigma}{n-2\sigma}}+\Lambda\cdot x\bar v^{\frac{n+2\sigma}{n-2\sigma}},
\]
where $\Lambda=\lim_{k\to\infty}\Lambda_k$. The Kazdan-Warner type identity in \cite{JLX} gives us
\[
\int_{\Sn}\nabla(P_{\sigma}(1)+\Lambda\cdot x)\nabla x\bar v^{\frac{2n}{n-2\sigma}}=0.
\]
It follows that $\Lambda=0$. Hence $\int_{\Sn}\bar vP_{\sigma}(\bar v)=P_{\sigma}(1)\int_{\Sn}\bar v^{\frac{2n}{n-2\sigma}}$. This together with \eqref{pe1} leads to $\dashint_{\Sn}\bar v^{\frac{2n}{n-2\sigma}}\geq 1$. On the other hand, $\dashint_{\Sn}\bar v^{\frac{2n}{n-2\sigma}}\leq \liminf\limits_{k\to\infty}v_k^{p_k}=1$. Hence
\[
\begin{cases}
&\dashint_{\Sn}\bar v^{\frac{2n}{n-2\sigma}}=1,\\
&\dashint_{\Sn}\bar vP_{\sigma}(\bar v)=P_{\sigma}(1).
\end{cases}
\]
This together with \eqref{eq:lim-k} leads to $v_k\to \bar v$ in $H^{\sigma}(\Sn)$. Clearly $\bar v\in\M_0^{\frac{2n}{n-2\sigma}}$ and hence $\bar v\equiv 1$.
In the following we will expand $I_a(v)$ for $v\in\M^p_0$ near $1$. Similar to Lemma \ref{lempe3},
\[
T_1\M^p_0=\mathrm{span}\{\mbox{spherical harmonics of degree } \geq 2\}.
\]
We need the following lemma which is a refined version of Lemma \ref{lempe4} and it can be proved in a similar way.
\begin{lem}\label{lempe4-p}
For $\tilde{w}\in T_1\M_0^p$, $\frac{2n-2\sigma}{n-2\sigma}\le p\le\frac{2n}{n-2\sigma}$, $\tilde{w}$ close to $0$, there exist $\mu(\tilde{w})\in \R$, $\eta(\tilde{w})\in \R^{n+1}$ being $C^2$ functions such that
\be\label{pe2-p}
 \dashint_{\mathbb{S}^n}|1+\tilde{w}+\mu+\eta\cdot x|^p=1
\ee
and
\be\label{pe3-p}
 \int_{\Sn}|1+\tilde{w}+\mu+\eta\cdot x|^px=0.
\ee
Furthermore, $\mu(0)=0, \eta(0)=0, D\mu(0)=0$ and $D\eta(0)=0$, and $\mu, \eta$ have uniform (with respect to $p$) $C^2$ modulo of continuity near $0$.
\end{lem}
As before we will use $\tilde w$ as local coordinates of $v\in \M^p_0$.
Let
\[
 \tilde{E}(\tilde{w})=I_a(v)=a\dashint_{\Sn} v P_{\sigma}(v)+(1-a)P_{\sigma}(1)\dashint_{\Sn} v^2,
\]
where $\tilde{w}\in T_1\M_0$ and $v=1+\tilde{w}+\mu(\tilde{w})+\eta(\tilde{w})\cdot x$ as in Lemma \ref{lempe4-p}.
Hence
\[
\tilde{E}(\tilde{w})=P_{\sigma}(1)(1+2\mu(\tilde{w}))+a\dashint_{\mathbb{S}^n}\tilde{w}P_\sigma(\tilde{w})+(1-a)P_{\sigma}(1)\dashint_{\Sn}\tilde w^2+o(\|\tilde{w}\|^2_{H^\sigma(\mathbb{S}^n)}).
\]
Since
\[
\mu(\tilde w)=-\frac{p-1}{2}\dashint_{\Sn}\tilde w^2+o(\|\tilde{w}\|^2_{H^\sigma(\mathbb{S}^n)}),
\]
we have
\[
\tilde{E}(\tilde{w})=P_{\sigma}(1)+a\dashint_{\mathbb{S}^n}\tilde{w}P_\sigma(\tilde{w})-(p-2+a)P_{\sigma}(1)\dashint_{\Sn}\tilde w^2+o(\|\tilde{w}\|^2_{H^\sigma(\mathbb{S}^n)}).
\]
For $a$ close to $1$ and $p$ close to $\frac{2n}{n-2\sigma}$, we have that $(p-2+a)P_{\sigma}(1)$ is close to $\frac{n+2\sigma}{n-2\sigma}P_{\sigma}(1)$, which is the first eigenvalue of $P_{\sigma}$. Similar to \eqref{pe5}, there exists some positive constant $C(n,\sigma)$ determined by the difference of the first and the second eigenvalues of $P_{\sigma}$ such that for $a$ close to $1$ and $p$ close to $\frac{2n}{n-2\sigma}$ we have
\[
a\dashint_{\mathbb{S}^n}\tilde{w}P_\sigma(\tilde{w})-(p-2+a)P_{\sigma}(1)\dashint_{\Sn}\tilde w^2\geq \frac{1}{C(n,\sigma)}\dashint_{\Sn}\tilde wP_{\sigma}(\tilde w),
\]
which leads to that for $k$ large we have $I_{a_k}(v_k)\geq P_{\sigma}(1)$. This is a contradiction.
\end{proof}

\appendix

\section{Bessel potential spaces and conformally invariant operators on spheres}\label{spaces}

In this section, we recall some results for $P_\sigma$ and Bessel potential spaces on spheres which can be found in \cite{PS}, \cite{R}, \cite{Stri} and \cite{T}.

Let $\Delta_{g_{\Sn}}$ be the Laplace-Beltrami operator on the standard sphere. For $s>0$ and $1< p<\infty$, the Bessel potential space $H^{s}_p(\Sn)$
is the set consisting of all functions $u\in L^p(\Sn)$ such that
$(1-\Delta_{g_{\Sn}})^{s/2}u\in L^p(\Sn)$, with the norm $\|u\|_{H^{s}_p(\Sn)}:=\|(1-\Delta_{g_{\Sn}})^{s/2}u\|_{L^p(\Sn)}$. When $p=2$, $H_2^{\sigma}(\Sn)$ coincides with the Hilbert space $H^{\sigma}(\Sn)$ which is the closure of $C^{\infty}(\Sn)$ under the norm
\[
\|u\|_{H^{\sigma}(\Sn)}:=\int_{\Sn}vP_{\sigma}v \,\ud vol_{g_{\Sn}}
\]
with equivalent norms.

If $sp<n$, then the embedding $H^{s}_p(\Sn) \rightarrow L^{\frac{np}{n-sp}}(\Sn)$ is continuous and the embedding $H^{s}_p(\Sn) \hookrightarrow L^q(\Sn)$ is compact for $q<\frac{np}{n-sp}$. If $0<s-\frac{n}{p}<1$, then the embedding $H^{s}_p(\Sn) \rightarrow C^{s-\frac{n}{p}}(\Sn)$ is continuous.
%Let $u\in H^{s}_p(\Sn)$ with $sp<n$, then $u\in L^{\frac{np}{n-sp}}(\Sn)$. Moreover,
%\[
% \|u\|_{ L^{\frac{np}{n-sp}}(\Sn)}\leq C \|u\|_{H^{s}_p(\Sn)},
%\]
%where $C>0$ depending only on $n,s, p$. It is also well-known that for any $q<\frac{np}{n-sp}$ the embedding $H^{s}_p(\Sn) \hookrightarrow L^q(\Sn)$ is compact.

It is also well-known (see, e.g., \cite{Mo}) that $P_{\sigma}$ is the inverse of the spherical Riesz potential
\be\label{P sigma inverse}
R_{2\sigma}(f)(\xi)=\frac{\Gamma(\frac{n-2\sigma}{2})}{2^{2\sigma}\pi^{n/2}\Gamma (\sigma)}\int_{\Sn}\frac{f(\zeta)}{|\xi-\zeta|^{n-2\sigma}}\,\ud vol_{g_{\Sn}}(\zeta),\quad f\in L^p(\Sn).
\ee
%where $1< p<\infty$ .
%\marginpar{check p and add samko }
%Note that
%\[
 %P_{s}^{-1}(\phi) =\frac{1}{c(n,s)}\int_{\Sn}\frac{\phi(\eta)}{|\xi-\eta|^{n-s}}\,\ud \eta,
%\]
%for $\phi\in L^p(\Sn)$.

\begin{prop}[Pavlov and Samko \cite{PS}] \label{norms are same}
For any function $u\in L^p(\Sn)$, then $u\in H^{s}_p(\Sn)$ if and only if there exists a function $v\in L^p(\Sn)$ such that
$u=R_{s}(v)$. Moreover, there exists a positive constant $C_1$ depending only on $n,s,p$ such that
\[
\frac{1}{C_1} \|u\|_{H^{s}_p(\Sn)}\leq \|v\|_{L^p(\Sn)}\leq C_1 \|u\|_{H^{s}_p(\Sn)}.
\]
\end{prop}

\bigskip

\noindent T. Jin

\noindent Department of Mathematics, Rutgers University\\
110 Frelinghuysen Road, Piscataway, NJ 08854, USA

\smallskip
\noindent\emph{Current address:}

\smallskip

\noindent Department of Mathematics, The University of Chicago\\
5734 S. University Avenue, Chicago, IL, 60637 USA\\[1mm]
Email: \textsf{tj@math.uchicago.edu}

\bigskip

\noindent Y.Y. Li

\noindent Department of Mathematics, Rutgers University\\
110 Frelinghuysen Road, Piscataway, NJ 08854, USA\\
Email: \textsf{yyli@math.rutgers.edu}

\bigskip

\noindent J. Xiong

\noindent School of Mathematical Sciences, Beijing Normal University\\
Beijing 100875, China

\smallskip
\noindent\emph{Current address:}

\smallskip

\noindent Beijing International Center for Mathematical Research, Peking University\\
Beijing 100871, China\\[1mm]
Email: \textsf{jxiong@math.pku.edu.cn}

\end{document}